\documentclass[12pt]{amsart}

\usepackage{xcolor}
\RequirePackage[all]{xy}

\newtheorem{thm}{Theorem}[section]
\newtheorem{cor}[thm]{Corollary}
\newtheorem{lem}[thm]{Lemma}
\newtheorem{prop}[thm]{Proposition}
\theoremstyle{definition}
\newtheorem{defn}[thm]{Definition}
\newtheorem{rem}[thm]{Remark}

\newtheorem{ex}[thm]{Example}
\numberwithin{equation}{section}

\theoremstyle{plain}

\usepackage{amsfonts}
\usepackage{amssymb}
\usepackage{color}
\usepackage{mathtools}
\usepackage{bbm}
\newcommand{\be}{\begin{equation}}
	\newcommand{\en}{\end{equation}}
\newcommand{\Lc}{{\mc L}}
\newcommand{\bei}{\begin{itemize}}
	\newcommand{\eni}{\end{itemize}}
\newcommand{\ip}[2]{\langle{#1}|{#2}\rangle}

\newcommand{\C}{\mathfrak{C}}

\numberwithin{equation}{section}

\newcommand{\mb}{\mathbb}
\newcommand{\mc}{\mathcal}
\newcommand{\eul}{\mathfrak}
\newcommand{\A}{\eul A}
\newcommand{\Ao}{{\eul A}_{\scriptscriptstyle 0}}

\newcommand{\D}{\mc D}

\newcommand{\QA}{{\mathcal Q}_{\Ao}^{\,\C}(\A)}
\newcommand{\IA}{{\mathcal I}_{\Ao}^{\,\C}(\A)}

\newcommand{\id}{{\sf e }}

\newcommand{\DD}{\mathfrak D}

\newcommand{\B}{{\eul B}}

\newcommand{\ad}{^{\mbox{\scriptsize $\dag$}}}

\newcommand{\X}{{\mathfrak X}}

\newcommand{\mult}{\,{\scriptstyle \square}\,}
\newcommand{\vp}{\Phi}

\newcommand{\LXH}{{\mathcal L}\ad(\DD,\X)}
\newcommand{\LDH}{{\mathcal L}\ad(\D,\H)}
\newcommand{\LDXr}{{\mathcal L}\ad(\DD_\pi,\X)}
\newcommand{\idop}{{\mb I}}

\newcommand{\LD}{{\mc L}\ad(\D)}
\newcommand{\LDb}{{\mc L}\ad(\D)_b}

\numberwithin{equation}{section}
\newcommand{\cls}[1]{\Lambda_\vp(#1)}
\newcommand{\ppi}{\Pi}
\def\H{\mc H}
\def\x{\relax\ifmmode {\mbox{*}}\else*\fi}

\newcommand{\BH}{\mc{B}(\mathcal{H})}
\newcommand{\ze}{_{\scriptscriptstyle 0}}
\newcommand{\up}{\raisebox{0.7mm}{$\upharpoonright$}}

\newcommand{\lcqa}{locally convex quasi *-algebra }

\newcommand{\vertiii}[1]{{\left\vert\kern-0.25ex\left\vert\kern-0.25ex\left\vert #1 
					\right\vert\kern-0.25ex\right\vert\kern-0.25ex\right\vert}}
\definecolor{Magenta}{rgb}{1,0,1}

\newcommand{\BibTeX}{B\kern-0.1emi\kern-0.017emb\kern-0.15em\TeX}
\newcommand{\XYpic}{$\mathrm{X\kern-0.3em\raisebox{-0.18em}{Y}}$-$\mathrm{pic}\,$}

\newcommand{\ed}{\end{document}}
\begin{document}

%-------------------------------------------------------------------------
% editorial commands: to be inserted by the editorial office
%
%\firstpage{1} \volume{228} \Copyrightyear{2004} \DOI{003-0001}
%\seriesextraline{This is the Concrete Title of this Book\br H.E. R and S.T.C. W, Eds.}
%\firstpage{1}

	\title[GNS-construction for positive C*-valued sesquilinear maps]{GNS-construction for positive C*-valued sesquilinear maps on a quasi *-algebra}
%----------Author 1
	\author[G. ~Bellomonte]{Giorgia Bellomonte}

%\author[Birkh\"auser \textit{et al.}]%
%{Birkh\"{a}user Publishing Ltd.}
%
\address{Dipartimento di Matematica e Informatica, Universit\`a degli Studi  di Palermo, Via Archirafi n. 34,  I-90123 Palermo, Italy}
\email{giorgia.bellomonte@unipa.it} 

%
%\thanks{This file has been typeset with the option \texttt{draft} to illustrate that feature and its purpose.}

%----------Author 2
\author[S. ~Ivkovi\'{c}]{Stefan Ivkovi\'{c}}

\address{Mathematical Institute of the Serbian Academy of Sciences and Arts, Kneza Mihaila 36, 11000 Beograd, Serbia}
\email{stefan.iv10@outlook.com}
%----------Author 3

\author[C. ~Trapani]{Camillo Trapani}
%\author[]{Rafa\l \ Ab\l amowicz}
\address{Dipartimento di Matematica e Informatica, Universit\`a degli Studi  di Palermo, Via Archirafi n. 34,  I-90123 Palermo, Italy}
\email{camillo.trapani@unipa.it} 

\date{\today}
%\dedicatory{Last Revised:\\ \today}
%%% ----------------------------------------------------------------------
\begin{abstract}
The GNS construction for positive invariant sesquilinear forms on  quasi *-algebras $(\A,\Ao)$ is generalized to a class of positive sesquilinear maps from $\A\times \A$ into a C*-algebra $\C$. The result is a *-representation taking values in a space of operators acting on a certain quasi normed $\C$-module.
\end{abstract}
\label{page:firstblob}
%%% ----------------------------------------------------------------------
\maketitle
%%% ----------------------------------------------------------------------
%\tableofcontents

	\section{Introduction and  Basic definitions} 
The Gelfand-Naimark-Segal construction is nowadays a common tool when studying the structure properties of locally convex *-algebras, since it provides *-representations of the given *-algebra into some space of  operators acting in Hilbert space.

The basic idea consists in building up *-representations starting from a positive linear functional on a *-algebra, constructing a Hilbert space from it and then defining operators in natural way using the multiplication of the given *-algebra.

This construction was given first in the case of C*-algebras and produces bounded operators in Hilbert spaces, but the paper of Powers \cite{powers}, in the early 1970's, puts in evidence its generality if one is willing to pay the price of dealing with *-algebras of unbounded operators.
Since then this procedure has been generalized in very many directions and in very many ways: extensions to the case of partial *-algebras and quasi *-algebras have been considered, see \cite{ait_book} and \cite{FT_book}.
In particular, it has appeared clear that, when dealing with algebraic structures where the multiplication is only partially defined, it is convenient to replace positive linear functionals with positive sesquilinear forms enjoying certain {\em invariance} properties.

In this paper, we will analyze the possible generalization of the GNS construction for a quasi *-algebra $(\A, \Ao)$, see below for a formal definition, starting from a {\em positive sesquilinear (i.e., conjugate-bilinear) map} $\Phi$ taking its values in a C*-algebra $\C$.
In this case one expects that the image of a *-representation  is a space of operators acting on some Hilbert C*-module. As we will see in what follows this is not always the case (this depends on the  Cauchy-Schwartz-like inequality $\Phi$ satisfies) and for this reason we have introduced quasi-Banach spaces whose norm is defined by a $\C$-valued inner product, named, for short, quasi $B_\C$-spaces. 

Positive and completely positive maps on C*-algebras or Operator algebras play an important role  in many
applications such as quantum theory, quantum information, quantum probability
theory, and a lot of deep mathematical results have been obtained (see, e.g. \cite{stormer}). On the other hand, it is now long time that the C*-algebraic approach to quantum theories has
been considered as too rigid framework where casting all objects
of physical interest. For this reason several possible generalizations have been proposed:
quasi *-algebras, partial *-algebras and so on. They reveal in fact to be more suited to cover situations where unbounded operator algebras occur.
These facts provide, in our opinion, good motivations for the generalizations we are proposing here.\\

The paper is organized as follows.
In Section \ref{sec: ips appli} we analyze some properties of positive sesquilinear $\C$-valued maps, the {\em quasi-inner product} it defines on a given vector space $\X$ and study in particular the quasi $B_\C$-space it generates. Section \ref{sect_GNS_rep} is devoted to the construction of the *-representation associated to $\Phi$. This is in fact a generalization of Paschke result \cite{paschke} which is the first involving Hilbert C*-modules (as far as we know). The proofs we give are often adaptations to the case under consideration of the corresponding ones for positive sesquilinear {\em forms} but is not surprising at all, since all generalizations of the GNS representation  are variants of the beautiful construction made by Gelfand, Naimark and Segal. The main results of the paper are Theorem \ref{thm_rep} and Corollary \ref{moreGNS alg} which provides a representation of positive C*-valued maps on unital *-algebras.  Moreover, Corollary \ref{stinesp}, and Corollary \ref{cor: 3.12} illustrate also the applications to positive linear C*-valued maps on (quasi) *-algebras.
Examples coming mostly from the theory of noncommutative integration are discussed.

\medskip
To keep the paper sufficiently self-contained we begin with some preliminary definitions and facts.

%Quasi *-algebras were introduced in the 1980's  by Lassner \cite{Lass2, Lass3} with the scope
%of providing a convenient mathematical
%environment where one could properly deal with the thermodynamical limit
%of local observables of certain quantum statistical models that
%were not covered by the algebraic formulation of
%quantum theories proposed by Haag and Kastler \cite{Haag}. For this
%purpose, %of course, a topological structure with sufficiently many
%%reasonable properties is needed; in other terms, 
%{\it locally convex} quasi *-algebras have to be considered \cite{Ant1,
	%	Trap3}. 

\medskip  
A {\em quasi *-algebra} $(\A, \Ao)$ is a pair consisting of a vector space $\A$ and a *-algebra $\Ao$ contained in $\A$ as a subspace and such that
\begin{itemize}
	\item $\A$ carries an involution $a\mapsto a^*$ extending the involution of $\Ao$;
	\item $\A$ is  a bimodule over $\A_0$ and the module multiplications extend the multiplication of $\Ao$. In particular, the following associative laws hold:
	\begin{equation}\notag \label{eq_associativity}
		(xa)y = x(ay); \ \ a(xy)= (ax)y, \quad \forall \ a \in \A, \  x,y \in \Ao;
	\end{equation}
	\item $(ax)^*=x^*a^*$, for every $a \in \A$ and $x \in \Ao$.
\end{itemize}

The
\emph{identity} of $(\A, \Ao)$, if any, is a necessarily unique element $\id\in \Ao$, such that
$a\id=a=\id a$, for all $a \in \A$.

We will always suppose that
\begin{align*}
	&ax=0, \; \forall x\in \Ao \Rightarrow a=0 \\
	&ax=0, \; \forall a\in \A \Rightarrow x=0. 
\end{align*}
Clearly, both these conditions are automatically satisfied if $(\A, \Ao)$ has an identity $\id$.
\begin{defn}
A  quasi *-algebra $(\A, \Ao)$ is said to be  {\em locally convex} if $\A$ is a locally convex vector space, with a topology $\tau$ enjoying the following properties
\begin{itemize}
	\item[{\sf (lc1)}] $x\mapsto x^*$, \ $x\in\A\ze$,  is continuous;
	\item[{\sf (lc2)}] for every $a \in \A$, the maps  $x \mapsto ax$ and
	$x \mapsto xa$, from $\A\ze$ into $\A$, $x\in \A\ze$, are continuous;
	\item[{\sf (lc3)}] $\overline{\Ao}^\tau = \A$; i.e., $\Ao$ is dense in $\A[\tau]$.
\end{itemize}
The involution of $\Ao$ extends by continuity to $\A$. 
Moreover, if $\tau$ is a norm topology, with norm $\|\cdot\|$, and
\begin{itemize}
	\item[{\sf (bq*)}] $\|a^*\|=\|a\|, \; \forall a \in \A$
\end{itemize}
then, $(\A, \Ao)$ is called a {\em normed  quasi *-algebra} and a {\em Banach  quasi *-algebra} if the normed vector space $\A[\|\cdot\|]$ is complete.
\end{defn}

{The simplest example of a \lcqa  is obtained
	by taking the completion $\A:= \widetilde{\A_0}[\tau]$ of a locally convex *-algebra
	$\A_0[\tau]$}
with separately (but not jointly) continuous multiplication (this was, in fact, the case considered at an early stage of the theory, in view of applications to quantum physics).

\medskip 

In the whole paper, $\C$ will denote a C*-algebra with unit ${\it 1}$ and norm $\|\cdot\|_\C$ and $\C^+$ its positive cone.
If $\omega$ is a continuous linear functional on $\C$, we denote by $\|\omega\|_\C^*$ the norm in the Banach dual of $\C$.
Let ${\mc S}(\C)$ denote the set of all positive linear functionals on $\C$ such that $\|\omega\|_\C^*=1$. We recall that \begin{equation*}\label{eqn_norm_one}\|z\|_\C^2 = \sup_{\omega\in {\mc S}(\C)}\omega(z^*z).\end{equation*}
In particular, if $z$ is a normal element of $\C$, 
\begin{equation}\label{eqn_norm_normal}
	\|z\|_\C = \sup_{\omega\in {\mc S}(\C)}|\omega(z)|.
\end{equation}
Hence, if $\C$ is a commutative C*-algebra,
\begin{equation}\label{eqn_norm_two}\|z\|_\C = \sup_{\omega\in {\mc S}(\C)}|\omega(z)|, \quad \forall z\in \C.\end{equation}

%{\color{blue}Let us examine the case on a   *-algebra $\A$ and a positive linear map $\theta:\A \to\C$ in the  sense that $\theta(a^*a)\in \C^+$ for every $a \in \A$. The latter is linked to a sesquilinear map: $\vp$: $\vp(a,b)= \theta(b^*a)$; do we get some form of GNS? TO BE WRITTEN}
\section{Positive  sesquilinear $\C$-valued maps}
\label{sec: ips appli} In this section we will study {\em positive  sesquilinear $\C$-valued maps on   $\X\times\X$} when $\X$ is either simply a vector space or a right (left) module on $\C$, or a locally convex quasi *-algebra which is a $\C$-module. Throughout the section we progressively add some hypotheses on $\vp$ to get more results.
\subsection{The case of a vector space}
Let $\X$ be a complex vector space  %and with norm $\|\cdot\|_\C$ and $\id$ be the unit of $\X$. 
and $\vp$ a  positive  sesquilinear $\C$-valued map on   $\X\times\X$  $$\vp:(a,b)\in\X\times\X\to\vp(a,b)\in\C;$$ i.e., a map with the properties  \begin{itemize}
	\item[$i)$] $\vp(a,a)\in\C^+$,
	\item[$ii)$]$\vp(\alpha a+\beta b,\gamma c)=\overline{\gamma}[\alpha\vp( a,c)+\beta \vp(b,c)]$,  
\end{itemize}
with $a,b,c \in\X$ and $\alpha,\beta,\gamma\in\mathbb{C}$. \\
The positive  sesquilinear $\C$-valued map $\vp$ is called {\em faithful} if $$\vp(a,a)=0 \;\Rightarrow\; a=0.$$

\smallskip
By property $i)$ it follows that \begin{itemize}
	\item[$iii)$]  $\vp(b,a)=\vp(a,b)^*$, for all $a,b\in\X$.
\end{itemize} In fact,
let $\alpha\in\mathbb{C}$ and $a,b\in\X$, then
\begin{equation*}
	0\leq \vp(a+\alpha b, a +\alpha b)=\vp(a,a)+|\alpha|^2\vp(b,b)+ \alpha\vp(a,b)+\overline{\alpha}\vp(b,a)
\end{equation*}
Since $\vp(a+\alpha b, a +\alpha b)$, $\vp(a,a)$ and $\vp(b,b)$ are positive hence hermitian, so it is $\alpha\vp(a,b)+\overline{\alpha}\vp(b,a)$; if we choose $\alpha=1$ and $\alpha=i$ we get both $$\vp(a,b)+\vp(b,a)=(\vp(a,b)+\vp(b,a))^*=\vp(a,b)^*+\vp(b,a)^*$$
and $$i\vp(a,b)-i\vp(b,a)=(i\vp(a,b)-i\vp(b,a))^*=-i\vp(a,b)^*+i\vp(b,a)^*$$ hence $$\vp(a,b)-\vp(b,a)=-\vp(a,b)^*+\vp(b,a)^*$$ if we add the first and the third equality we get $\vp(a,b)=\vp(b,a)^*$.

\begin{defn} \label{defn_CSmap} Let $\vp$ be a positive sesquilinear $\C$- valued map. We say that $\vp$ satisfies a Cauchy-Schwarz  inequality if \begin{equation}\label{eq: CS} \|\vp(a,b)\|^2_\C \leq \|\vp(a,a)\|_\C\|\vp(b,b)\|_\C, \quad \forall a,b \in \X.\end{equation}
\end{defn}
\begin{ex} \label{ex_zero}  
Let $\X=\C$ and define
$$\vp(a,b)= b^*a .$$
It is clear that $\vp$ is a  positive sesquilinear map  of $\C\times \C$ into $\C$. $\vp$ satisfies \eqref{eq: CS}: 
\begin{align*}
	\|\vp(a,b)\|_\C^2 &=\|b^*a\|_\C^2 \leq \|b\|_\C^2\|a\|_\C^2\\ &= \|a^*a\|_\C\|b^*b\|_\C= \|\vp(a,a)\|_\C\|\vp(b,b)\|_\C, \quad \forall a,b \in \C.
\end{align*} 
\end{ex}

\begin{lem}\label{lemma 1} Let $\vp$ be a positive  sesquilinear $\C$-valued map $\vp$ on   $\X\times\X$.  Then,
	\begin{itemize}
		\item[(i)] for all $a,b\in\X$, $$\|\vp(a,b)\|_\C\leq2\|\vp(a,a)\|_\C^{1/2}\|\vp(b,b)\|_\C^{1/2}.$$ 
		\item[(ii)]If $\C$ is commutative, then $\vp$ satisfies the Cauchy-Schwarz inequality.
	\end{itemize}  
\end{lem}
\begin{proof}
	Let  $\omega$ be a positive linear functional on $\C$ and let $\varphi:\X\times \X\to\mathbb{C}$ be given by $$\varphi(a,b)=\omega(\vp(a,b)),\quad \forall a,b\in\X.$$ 
	Since $\vp$ is sesquilinear and positive and by linearity  and positivity of $\omega$, it follows that $\varphi$ is a positive  sesquilinear form  on  $\X\times \X$.  Hence, the classical Cauchy-Schwarz inequality holds true: for all $a,b\in\X$ we have that 	$$|\omega(\vp(a,b))|^2\leq\omega(\vp(a,a))\omega(\vp(b,b)),\quad  \forall a,b\in\X.$$
	Then, by \eqref{eqn_norm_normal}, taking the supremum over $\omega\in {\mc S}(\C)$, we get
	the inequality
	\begin{equation*}\label{eqn_ineq_main}
		|\omega(\vp(a,b))|\leq \|\vp(a,a)\|_\C	\|\vp(b,b)\|_\C.
	\end{equation*}
	If $\C$ is commutative, using \eqref{eqn_norm_two}, we get 
	$$\|\vp(a,b)\|_\C\leq \|\vp(a,a)\|_\C^{1/2}\|\vp(b,b)\|_\C^{1/2},  \quad \forall a,b\in\X.$$
	This proves (ii).\\
	Let us come back to the general case.
	Without loss of generality, we can consider $\C$ as a C*-subalgebra of $\B(\H)$ (for some Hilbert space $\H$); thus, for all $x\in\H$,  \begin{eqnarray*}
		|\ip{\vp(a,b)x}{x}|^2&\leq&\ip{\vp(a,a)x}{x}\ip{\vp(b,b)x}{x}\\&\leq&\|\vp(a,a)\|\|\vp(b,b)\|\|x\|_\H^4.
	\end{eqnarray*} Let now $x,y\in\H$ with $\|x\|_\H=\|y\|_\H=1$. Then,  by the polarization identity
	\begin{eqnarray*}
		|\ip{\vp(a,b)x}{y}|&=&\frac{1}{4}\left|\sum_{i=0}^3i^k\ip{\vp(a,b)(x+i^ky)}{x+i^ky}\right|\\&\leq&\frac{1}{4}\sum_{k=0}^3\left|\ip{\vp(a,b)(x+i^ky)}{x+i^ky}\right|\\&\leq&\frac{1}{4}\sum_{k=0}^3\sqrt{\|\vp(a,a)\|\|\vp(b,b)\|}\|x+i^ky\|_\H^2\\&\leq&\sqrt{\|\vp(a,a)\|\|\vp(b,b)\|}(\|x\|_\H^2+\|y\|_\H^2)\\&=&2\sqrt{\|\vp(a,a)\|\|\vp(b,b)\|},
	\end{eqnarray*} since $\sum_{k=0}^3\|x+i^ky\|_\H^2=4(\|x\|_\H^2+\|y\|_\H^2)$. Taking now the supremum over all unit vectors $x,y\in\H$, we get $$\|\vp(a,b)\|_\C\leq 2\|\vp(a,a)\|_\C^{1/2}\|\vp(b,b)\|_\C^{1/2}.\qedhere$$
\end{proof}

The Stinespring theorem \cite[Theorem 1.2.7]{stormer} yields a inequality for positive  linear $\C$-valued maps on  C*--algebras, see  \cite[Theorem 1.3.1]{stormer}. Motivated by that result, we provide the following corollary.

\begin{cor}\label{stinesp}
	Let $\A$ be a *--algebra with unit
	$\id$ and let $\omega$ be a positive  linear $\C$-valued map on $\A$. Then,  $$4\|\omega(\id)\|_\C\,\|\omega(a^*a)\|_\C\geq\|\omega(a)\|_\C^2=\|\omega(a^*)\|_\C\,\|\omega(a)\|_\C, \quad\forall a\in\A.$$\end{cor}
\begin{proof}It suffices to apply Lemma \ref{lemma 1} to $\vp(a,b)=\omega(b^*a)$, $a,b\in\A$.
%	By Corollary \ref{moreGNS alg}, we have that $\omega(a)=\ip{\ppi_\omega(a)\eta_\omega}{\eta_\omega}_{\X_\vp}$ for every $a\in\A$.
%	Since $\ppi_\omega$ is a *-representation, we get, for all $a\in\A$:$$\omega(a^*)=\ip{\ppi_\omega(a^*)\eta_\omega}{\eta_\omega}_{\X_\vp}=\ip{\eta_\omega}{\ppi_\omega(a)\eta_\omega}_{\X_\vp}=\left(\ip{\ppi_\omega(a)\eta_\omega}{\eta_\omega}_{\X_\vp}\right)^*,$$ and $$\omega(a^*a)=\ip{\ppi_\omega(a)\eta_\omega}{\ppi_\omega(a)\eta_\omega}_{\X_\vp}$$ hence $$\|\omega(a^*a)\|_\C=\|\|\ppi_\omega(a)\eta_\omega\|_{\X_\vp}^2\|_\C.$$ By Lemma \ref{lemma 1}, we have that, for every $a\in\A$: $$\|\ip{\ppi_\omega(a)\eta_\omega}{\eta_\omega}_{\X_\vp}\|_\C^2\leq 4\|\|\eta_\omega\|_{\X_\vp}^2\|_\C\|\|\ppi_\omega(a)\eta_\omega\|_{\X_\vp}^2\|_\C,$$ hence
%	$$\|\omega(a^*)\|_\C\,\|\omega(a)\|_\C=\|\omega(a)\|_\C^2\leq 4\|\|\eta_\omega\|_{\X_\vp}^2\|_\C\|\|\ppi_\omega(a)\eta_\omega\|_{\X_\vp}^2\|_\C.$$Now recall from our GNS-construction that $\eta_\omega=\id+N_\vp$, where $\vp(a,\id)=\omega(a)$ for all $a\in\A$, hence $$\ip{\eta_\omega}{\eta_\omega}_{\X_\vp}=\vp(\id,\id)=\omega(\id)$$ therefore $\|\|\eta_\omega\|_{\X_\vp}^2\|_\C=\|\omega(\id)\|_\C$. This concludes the proof. 
	\end{proof}
	
	\begin{rem}
		If $\vp(a,b)=\omega(b^*a)$, $a,b\in\A$ satisfies the Cauchy Schwarz inequality in the norm (e.g. if either $\C$ is commutative or if $\A$ is a $\C$-module and $\omega$ is $\C$-linear) then$$\|\omega(\id)\|_\C\,\|\omega(a^*a)\|_\C\geq\|\omega(a^*)\|_\C\,\|\omega(a)\|_\C, \quad\forall a\in\A.$$ 
	\end{rem}

%{\gbc are we going to make a comment of what happens if $\vp$ is faithful?}

%\begin{rem}
%	The estimate of the previous lemma can be improved if we require either $\vp$ to be positive sesquilinear and $\C$-linear with the product of $\C$ (see Lemma \ref{lem: Cauchy-Schwarz 1}), or $\vp$ to be positive and sesquilinear with values on a  commutative C*-algebra $\C$ (see Lemma \ref{lem: Cauchy-Schwarz commut}).
%\end{rem}

%\begin{rem}\label{rem: defn semidefn}Both $\X$ and $\vp$ are said {\em faithful or faithful} if $\vp(a,a)=0$ is equivalent to $a=0$. They are said {\em not faithful} if 	 $\vp(a,a)=0$ does not imply that $a=0$.
%\end{rem}
\begin{defn}\label{defn: C*-valued quasi inner product}
Let $\X$ be a vector space. A  faithful positive  sesquilinear $\C$-valued map $\vp$ on   $\X\times\X$ is said to be a {\em C*-valued quasi inner product} and we often will write $\ip{a}{b}_\vp:=\vp(a,b)$,  $a,b\in\X$.
\end{defn}

Let $\vp$ be a faithful positive  sesquilinear $\C$-valued map    $\vp:\X\times\X\to \C$.  %,  $\X$ is a bimodule over $\C$ and $\ip{\cdot}{\cdot}$ is a $\C$-linear.
A C*-valued quasi inner product induces a quasi norm $\|\cdot\|_\vp$ on $\X$:  

\label{vp norm}  \begin{equation*}
\label{eq: norm phi def}
\|a\|_\vp:=\sqrt{\|\ip{a}{a}_\vp\|_\C}=\sqrt{\|\vp(a,a)\|_\C},\quad a\in\X.\end{equation*} 

\medskip
This means that
\begin{align} 	
& \| a\|_\vp\geq0, \qquad \forall a\in\X \mbox{ and }\| a\|_\vp=0\Leftrightarrow a=0,\nonumber\\
& \|\alpha a\|_\vp=|\alpha|\| a\|_\vp, \qquad \forall \alpha\in\mathbb{C}, a\in\X,\nonumber\\ 
&   \|a+b\|_\vp\leq\sqrt{2}(\| a\|_\vp+\| b\|_\vp), \qquad \forall a,b\in\X.\label{eqn:Q3}
\end{align}  Indeed, by Lemma \ref{lemma 1}:\label{page: quasi norm}
\begin{eqnarray*}
\|a+b\|_\vp^2&=&\|\vp(a+b,a+b)\|_\C\\&\leq&\|\vp(a,a)\|_\C+2\|\vp(a,b)\|_\C+\|\vp(b,b)\|_\C\\
&\leq&\|a\|_\vp^2+4\|a\|_\vp \, \|b\|_\vp+\|b\|_\vp^2\leq2(\|a\|_\vp+\|b\|_\vp)^2, \,\,\forall a,b\in\X.
\end{eqnarray*}

The space $\X$ is then a quasi normed space w.r.to the quasi norm $\|\cdot\|_\vp$.

\begin{defn} 
If $\X$ is complete w.r.to the quasi norm $\|\cdot\|$, $\X$ will be said a {\em quasi Banach space with $\C$-valued quasi inner product} or for short a {\em quasi $B_\C$-space}.
\end{defn}
%\begin{rem}
%	As we will see in the following, if $\X$ is a module over $\C$ and $\vp$ is $\C$-linear, then the C*-valued quasi inner product is in fact a {\em classical} C*-valued  inner product, making of $\X$ a pre-Hilbert module over $\C$.
%\end{rem}

Let $\X$ be a quasi $B_\C$-space and $\DD(X)$ a dense subspace of $\X$.
A linear map $X:\DD(X)\to \X$ is {\em $\vp$-adjointable} if there exists a linear map $X^*$ defined on a subspace $\DD(X^*)\subset \X$ such that
$$\vp(Xa,b)= \vp(a, X^*b), \quad \forall a \in \DD(X), b\in \DD(X^*).$$

Let $\DD$ be a dense subspace of $\X$ and let us consider the following families of linear operators acting on $\DD$:
\begin{align*}		{\LXH}&=\{X \mbox{ $\vp$-adjointable}, \DD(X)=\DD;\; \DD(X^*)\supset \DD\} \\	{\Lc^\dagger(\DD)}&=\{X\in \LXH: X\DD\subset \DD; \; X^*\DD\subset \DD\} \\	{\Lc^\dagger(\DD)_b} &=\{Y\in \Lc^\dagger(\DD); \, {Y} \mbox{ is bounded on $\DD$} \}.\end{align*}
%where $\overline{Y}$ denotes the closure of $Y$.
The involution in $\LXH$ is defined by		$X^\dag := X^*\upharpoonright \D$, the restriction of $X^*$, the adjoint of $X$, to $\DD$.

The sets $\LD$ and ${\Lc^\dagger(\DD)_b}$ are *-algebras.

\begin{rem} \label{rem_closable}If $X\in \LXH$ then $X$ is closable. By definition $X$ is adjointable. Let $X^*$ be its adjoint with domain $\DD(X^*)$. We prove that $X^*$ is closed. Indeed, suppose that $\{u_n\}$ is a sequence in $\DD(X^*)$ such that \\ \mbox{$\|u_n-u\|_\vp \to 0$} for some $u\in \X$ and $\|X^*u_n -v\|_\vp \to 0$ for some $v\in \X$. Clearly $\|u_n-u\|_\vp \to 0$ is equivalent to $\vp(u_n-u,u_n-u)\to 0$.
Then by Lemma \ref{lemma 1}, we get, for every $y\in \X$,
 $$\|\vp(u_n-u,y)\|_\C^2 \leq 4\| \vp(u_n -u, u_n-u)\|_\C\|\vp(y,y)\|_\C\to 0.$$
Hence, for every $z\in \D$
$$ \|\vp(u_n,Xz)\|_\C= \|\vp(X^*u_n,z)\|_\C\to \|\vp(u,Xz)\|_\C.$$
On the other hand,
$$ \|\vp(X^*u_n,z)\|_\C \to \|\vp(v,z)\|_\C.$$
These relations imply that $u \in \DD(X^*)$ and $X^*u=v$. Thus $X^*$ is closed.
Now apply this result to $X^{\dag*}$ to obtain a closed extension of $X$.

\end{rem}

\begin{rem} \label{rem_partialmult}
$\LXH$ is also a {\em partial *-algebra} \cite{ait_book}  with respect to the following operations: the usual sum $X_1 + X_2 $,
the scalar multiplication $\lambda X$, the involution $ X \mapsto X\ad := X^* \up {\DD}$ and the \emph{(weak)}
partial multiplication 
defined whenever there a  exists $Y\in \LXH$ such that
$$\vp(X_2 a,X_1b)= \vp (Ya,b), \quad \forall a,b \in \DD.$$
The element $Y$, if it exists, is unique. We put $Y=X_1\mult X_2$.
\end{rem}

\bigskip
If $\vp$ is not faithful, we can consider the set $$N_\vp=\{a\in\X:\, \vp(a,a)=0_\C\}.$$
%
%\begin{prop}\label{prop Stef 1}
%	$ N_{\varphi}$ is a subspace of $\A$.
%\end{prop}
%\begin{proof}
%	Let $\omega$ be a positive linear functional on $\C$ and let $\widetilde{\vp}:\A\times \A\to\mathbb{C}$ be given by $$\widetilde{\vp}(a,b)=\omega(\vp(a,b)),\quad \forall a,b\in\A.$$
%	Since $\vp$ is sesquilinear and positive and by linearity  and positivity of $\omega$, it follows that $\widetilde{\vp}$ satisfies Defn \ref{defn: sesq map}. Hence, for all $a,b\in\A$ we have that \begin{equation}\label{eq: CS}
%		|\widetilde{\vp}(a,b)|^2\leq\widetilde{\vp}(a,a)\widetilde{\vp}(b,b).\end{equation} Thus, if $a\in N_{\varphi}$, we deduce that $\widetilde{\vp}(a,b)=0$ for every $b\in\A$. Now, since $\C$ is a C*-algebra, by the Gel'fand and Naimark theorem it can be considered as a C*-subalgebra of $\B(\H)$ (with $\H$ an opportune Hilbert space), so $\vp(a,b)$ can be viewed as a bounded, linear operator on $\H$. If $\omega$ is given by $\omega(T)=\ip{Tx}{x}$ for all $T\in\C$ and some $x\in\H$, then $\omega$ is a positive linear functional. Hence, if $\vp(a,a)=0$ for some $a\in\A$, then by \ref{eq: CS} we deduce that $\ip{\vp(a,b)x}{x}=0$ for all $x\in \H$ and by the polarization identity $\vp(a,b)$ is the zero operator on $\H$. This shows that $ N_{\varphi}  = \big\{ a \in \A : \varphi(a,b) = 0_\C, \ \forall \
%	b \in \A \big\}$ is a subspace.\end{proof}

\begin{lem}\label{N phi subspace}
$ N_{\varphi}$ is a subspace of $\X$.
\end{lem}
\begin{proof}
%Let  $\omega$ be a positive linear functional on $\C$ and let $\widetilde{\vp}:\X\times \X\to\mathbb{C}$ be given by $$\widetilde{\vp}(a,b)=\omega(\vp(a,b)),\quad \forall a,b\in\X.$$ 
%Since $\vp$ is sesquilinear and positive and by linearity  and positivity of $\omega$, it follows that $\widetilde{\vp}$ It is a positive not faithful complex valued sesquilinear map on $\X\times \X$.  Hence, the classical Cauchy-Schwarz inequality hold true: for all $a,b\in\X$ we have that \begin{equation}\label{eq: CS1}
%	|\widetilde{\vp}(a,b)|^2\leq\widetilde{\vp}(a,a)\widetilde{\vp}(b,b).\end{equation} 
%...
$ N_{\varphi} = \big\{ a \in \X :
\varphi(a,a) = 0_\C \big\} = \big\{ a \in \X : \varphi(a,b) = 0_\C, \ \forall \
b \in \X \big\}$ is an easy consequence of Lemma \ref{lemma 1}.
\end{proof}
For the sake of simplicity, we denote by $\Lambda_\vp(a)$ the coset containing $a\in \X$; i.e., $\Lambda_\vp(a)=a+N_\vp$.
% $\vp$ is well-defined on $\X/N_\vp$. 

We define 
a positive  sesquilinear $\C$-valued map on   $\X/N_\vp\times\X/N_\vp$ as follows: \begin{equation*}\ip{\cdot}{\cdot}_\vp:\X/N_\vp\times\,\X/N_\vp\to\C\end{equation*} \begin{equation}\label{eq: inner pr phi semidef}
\ip{\Lambda_\vp(a)}{ \Lambda_\vp(b)}_\vp:=\vp(a,b)
\end{equation}The associated quasi norm is:\begin{equation}
\label{eq: norm phi semidef}
\|\Lambda_\vp(a)\|_\vp:=\sqrt{\|\vp(a,a)\|_\C},\quad a\in\X.\end{equation}

It is easy to check that
\begin{lem} $\Lambda_\vp(\X)$ is a quasi normed space.\end{lem}
%\begin{proof}
%	 It suffices to bypass the not faithfulness of $\vp$ by considering the following faithful 
%%Put $\|a+N_\vp\|_\vp:=\|\ip{a+N_\vp}{a+N_\vp}_\vp\|_\C^{1/2}=\|\vp(a,b)\|_\C^{1/2}$.\gbrem{we should find another symbol for the other norm, even if, if $\vp$ is faithful they coincide}
%%This is a quasi norm, indeed it is non negative and such that $\|a+N_\vp\|_\vp=0$ implies $a\in N_\vp$, it is  absolutely homogeneous $\|\alpha a+N_\vp\|_\vp=|\alpha|\| a+N_\vp\|_\vp$ for every $\alpha\in\mathbb{C}$ and every $a\in\X$ and $\|a+b+N_\vp\|_\vp\leq\sqrt{2}\| a+N_\vp\|_\vp+\| b+N_\vp\|_\vp$, for every $a,b\in\X$. The triangular  inequality holds true, by Lemma \ref{lemma 1}:
%%\begin{eqnarray*}
%%\|a+b+N_\vp\|_\vp^2&=&\|\vp(a+b,a+b)\|_\C\\&\leq&\|\vp(a,a)\|_\C+2\|\vp(a,b)\|_\C+\|\vp(b,b)\|_\C\\&\leq&\|a+N_\vp\|_\vp^2+4\|a+N_\vp\|_\vp \, \|b+N_\vp\|_\vp+\|b+N_\vp\|_\vp^2\\&\leq&2(\|a+N_\vp\|_\vp+\|a+N_\vp\|_\vp)^2.
%%\end{eqnarray*}
%\end{proof}
%{\color{red}	We can assume from now on that $\vp$ is not faithful since this case is more general.}\gbrem{otherwise we should write, among others, Prop. \ref{prop_nonsing} again}\\
Denote by $\widetilde{\X}$ the completion of $(\X/N_\vp,\|\cdot\|_\vp)$. \begin{rem}\label{rem: joint cont}
We can extend $\ip{\cdot}{\cdot}_\vp$ defined in \eqref{eq: inner pr phi semidef} to $\widetilde{\X}\times \widetilde{\X}$ by continuity, taking into account that,   $\ip{\cdot}{\cdot}_\vp$ is jointly continuous by Lemma \ref{lemma 1}.
\end{rem}
%	\|\ip{a+N_\vp}{b+N_\vp}_\vp\|_\C=\|\vp(a,b)\|_\C&\leq&2\|\vp(a,a)\|_\C^{1/2}\|\vp(b,b)\|_\C^{1/2}\\&=&2\|a+N_\vp\|_\vp\|b+N_\vp\|_\vp,
%\end{eqnarray*} for all $a,b\in\X$.\end{rem} Thus $$\|a\|_\vp=\|\ip{a}{a}_\vp\|_\C^{1/2}, \quad\forall a\in\widetilde{\X}.$$

%	{\gbc Until now the last part of this section  is useless.\\
%		
%If $\vp$ is faithful (or faithful, besides the scalar norm $\|a\|_\vp:=\sqrt{\|\vp(a,a)\|_\C}$, it is possible to define a $\C$-valued norm $|\cdot|_\vp:\X\to\C^+$:
%
%$$|a|_\vp:=\vp(a,a)^{1/2},\quad a\in\X.$$ Lastly, put $$|x|_\C:=(x^*x)^{1/2}, \quad x\in\C,$$ the modulus in the C*-algebra $\C$.\\
%
%The following inequality puts into a relationship the two norms and the modulus just recalled.	\begin{lem}
%	$|\vp(a,b)|_\C\leq\|a\|_\vp\,|b|_\vp,\,\,\forall a,b\in\X.$
%\end{lem}\begin{proof}
%	By Lemma \ref{lem: CSI} it is \begin{eqnarray*}
	%		|\vp(a,b)|_\C &=&\sqrt{\vp(b,a)\vp(a,b)}\\&\leq&(\|\vp(a,a)\|_{\C} \,\vp(b,b))^{1/2}=\|a\|_\vp\,|b|_\vp.
	%\end{eqnarray*}	\end{proof}}

	%		\begin{defn}\label{defn: sesq map}
		%			Let $\C$ be a *-algebra with unit $e$ and with norm $\|\cdot\|_\C$,  $\X$ be a $\C$-bimodule and $\id$ be the unit of $\X$. 
		%		A {\em positive COMPATIBLE (?) $\C$-valued sesquilinear map on   $\X\times\X$} $$\vp:(a,b)\in\X\times\X\to\vp(a,b)\in\C$$ is a map with the properties  \begin{itemize}
			%			\item[i)] $\vp(\alpha a+\beta b,c)=\alpha\vp( a,c)+\beta \vp(b,c)$
			%			\item[ii)] $\vp(x a,b)=x\vp(a,b)$ and $\vp( ax,b)=\vp(a,b)x$,	\item[iii)]	$\vp(a,a)\in\C^+$,
			%		\end{itemize}
		%		with $a,b,c \in\X$, $x\in\C$ and $\alpha,\beta\in\mathbb{C}$. 
		%		\end{defn}
	
	\subsection{The case of a  module over $\C$.}
	In this section $\X$ is a right module over $\C$ and $\vp$ will be a  positive sesquilinear $\C$-valued  map on   $\X\times\X$ such that
	% \begin{equation}\label{eq: N_vp right module}\vp(ax,ax)\leq\|\vp(a,a)\|_\C^2\, x^*x, \quad a\in\X, x\in \C.\end{equation}
		\begin{equation}\label{eq: N_vp right  module}\|\vp(ax,ax)\|_\C\leq\|\vp(a,a)\|_\C\, \|x\|^2_\C, \quad a\in\X, x\in \C.\end{equation} 
	
	%Let us suppose that $\vp$ is faithful positive $\C$-valued sesquilinear map    $\ip{\cdot}{\cdot}:\E\times\E\to \C$ (see Definition \ref{defn: sesq map} and Remark \ref{rem: defn semidefn}) both makes of $\E$ a pre-$\C$-Hilbert space and induces on $\E$ a quasi norm given by $$\|a\|_\X=\|\ip{a}{a}\|_\C^{1/2}, \quad a\in\X.$$	
	%
	%\gbrem{qui c'è da decidere se essere coerenti nella notazione della sez 2}
	%
	%By Lemma \ref{lemma 1}, it is a quasi norm: $$\|\ip{a}{b}\|_\C\leq4\|\ip{a}{a}\|_\C^{1/2}\|\ip{b}{b}\|_\C^{1/2}$$  for all $a,b\in\X$.

	%Let us suppose that $\vp$ is faithful positive  sesquilinear $\C$-valued map    $\vp:\X\times\X\to \C$ (see Definition \ref{defn: sesq map} and Remark \ref{rem: defn semidefn}),  $\X$ is a bimodule over $\C$.% and $\ip{\cdot}{\cdot}$ is a $\C$-linear.

	If $\vp$ is not faithful we have
	\begin{lem}\label{lem: 4 L}Let  $\vp$ be a positive  sequilinear $\C$-valued map on $\X\times\X$ satisfying \eqref{eq: N_vp right module}, then $\X/N_\vp[\|\cdot\|_\vp ]$ is a normed right C*-module over $\C$. \end{lem}
	\begin{proof}
		First we observe that, by \eqref{eq: N_vp right module}, if $a\in N_\vp$ and $x\in \C$ then $ax\in N_\vp$. This implies that $\Lambda_\vp(a)x=\Lambda_\vp(ax)$ for every $a\in \X$ and $x\in \C$. 
		Moreover we have \begin{equation*} \|\Lambda_\vp(ax)\|_\vp\leq \|\Lambda_\vp(a)\|_\vp\|x\|_\C, \quad a\in\X, x\in \C.
		\end{equation*}	Indeed, from \eqref{eq: N_vp right module}, we get  \begin{eqnarray*}
			\|\Lambda_\vp(ax)\|_\vp^2&=&\|\vp(ax,ax)\|_\C\leq\|\vp(a,a)\|_\C \|x\|_\C^2 \\ &=&\|\Lambda_\vp(a)\|_\vp^2\|x\|_\C^2,\quad a\in\X, x\in \C.\qedhere
		\end{eqnarray*} 
	\end{proof}
	
	%\begin{rem} \ctrem{I would omit reference to the left case: it makes heavier writing all this; maybe at the end we could add a remark}
	%	We could symmetrically suppose that $\X$ is a left C*-module over $\C$ and that $\vp$ satisfies \begin{equation}\label{eq: N_vp left module}\vp(xa,xa)\leq\|\vp(a,a)\|_\C^2\, xx^* , \quad a\in\X, x\in \C.\end{equation}
	%	 In this case $\X/N_\vp$ is a  normed left C*-module over $\C$.
	%\end{rem}
	
	%If property \eqref{eq: N_vp right module} (resp. \eqref{eq: N_vp left module}) holds, then the completion $\widetilde{\X}$ of $(\X/N_\vp,\|\cdot\|_\vp)$ is also a right (resp. left) Banach module over $\C$, indeed the right (resp. left) multiplication by elements in $\C$ can be extended by continuity to $\widetilde{\X}$: \begin{equation*}
		%	\|ax\|_\vp\leq\|a\|_\vp\|x\|_\C, \forall a\in \widetilde{\X}, x\in\C,
		%\end{equation*} (resp. $\|xa\|_\vp\leq\|a\|_\vp\|x\|_\C$,  for all $a\in \widetilde{\X}, x\in\C$).
		
		If property \eqref{eq: N_vp right module} holds, then the completion $\widetilde{\X}$ of $(\X/N_\vp,\|\cdot\|_\vp)$ is also a right  Banach module over $\C$, indeed the right  multiplication by elements in $\C$ can be extended by continuity to $\widetilde{\X}$: \begin{equation*}
			\|ax\|_\vp\leq\|a\|_\vp\|x\|_\C,\quad \forall a\in \widetilde{\X}, x\in\C.
		\end{equation*}

		%\subsubsection{$\C$-linear positive  sesquilinear $\C$-valued maps} In this section 

		\begin{defn} \label{def_Clinear}
		Let $\vp$ be a  positive sesquilinear $\C$-valued  map on   $\X\times\X$. The map  $\vp$ is {\em $\C$-linear} if $$\vp(a,bx)=\vp(a,b)x, \; \forall x\in \C; a,b\in \X.$$ \end{defn}
		
		Then $\vp$ satisfies satisfies the Cauchy-Schwarz inequality as shown in \cite[Section 1.2]{MT}.

\begin{rem} 
If $\vp$ is a $\C$-linear positive  sesquilinear $\C$-valued map on   $\X\times\X$ then \eqref{eq: N_vp right module} holds.
%both \eqref{eq: N_vp right module} and \eqref{eq: N_vp left module} hold.	
Indeed, 	recalling that if $c\in\C^+$ then $t^*ct\leq\|c\|_\C t^*t$, $t\in\C$\begin{eqnarray*}\vp(ax,ax)&=&\vp(ax,a)x=\vp(a,ax)^* x=(\vp(a,a)x)^* x\\&=&x^*\vp(a,a)x\leq\|\vp(a,a)\|_\C\, x^*x,\end{eqnarray*} and recalling that the norm in a C*-algebra preserves the order on positive elements,  we get \eqref{eq: N_vp right module}.
% and 
%	\begin{equation*}\vp(xa,xa)=x\vp(a,a)x^*\leq\|\vp(a,a)\|_\C\, xx^*.\end{equation*}
In this case, in fact, $\X/N_\phi$ is a pre-Hilbert $\C$-module (see \cite[Definition 1.2.1]{MT}).
\end{rem}

\begin{rem}%\label{lem: CSI} 
	Let  $\vp$ be a $\C$-linear positive  sequilinear $\C$-valued map on $\X\times\X$, then $$\vp(b,a)\vp(a,b)\leq\|\vp(a,a)\|_{\C} \,\vp(b,b), \quad \forall a,b\in\X.$$ This is another generalization of the Cauchy-Schwarz inequality, see  \cite[Proposition 1.2.4]{MT}.
\end{rem}
%\begin{proof} See the proof of \cite[Proposition 1.2.4]{MT}.
%	Let $a\in\X$ with $\|\vp(a,a)\|_\C>0$. If $x\in\C$, then \begin{eqnarray*}
	%		0_\C&\leq&\vp(xa-b,xa-b)\\ &=&x\vp(a,a)x^*-x\vp(a,b)-\vp(b,a)x^*+\vp(b,b)\\
	%		&\leq&\|\vp(a,a)\|_\C xx^*-x\vp(a,b)-\vp(b,a)x^*+\vp(b,b),
	%	\end{eqnarray*}
%	since if $c\in\C^+$ then $t^*ct\leq\|c\|_\C t^*t$. If now we take $x=\frac{\vp(b,a)}{\|\vp(a,a)\|_\C}$ we get the desired inequality. If $\vp(a,a)=0_\C$ and $\vp(b,b)\neq0_\C$ then we can change the roles of $a$ and $b$ in the previous argument. If $\vp(a,a)=\vp(b,b)=0_\C$ then $\vp(a,b)=0_\C$. Indeed, $$0_\C\leq\vp(a+b,a+b)=\vp(a,a)+\vp(a,b)+\vp(b,a)+\vp(b,b)=\vp(a,b)+\vp(a,b)^*$$ and $$0_\C\leq\vp(a-b,a-b)=\vp(a,a)-\vp(a,b)-\vp(b,a)+\vp(b,b)=-(\vp(a,b)+\vp(a,b)^*)$$ then, $\vp(a,b)+\vp(a,b)^*=0_\C$, moreover it is both
%	$$0_\C\leq\vp(a+ib,a+ib)=\vp(a,a)-i\vp(a,b)+i\vp(b,a)-\vp(b,b)=-i(\vp(a,b)-\vp(a,b)^*)$$ and $$0_\C\leq\vp(a-ib,a-ib)=\vp(a,a)+i\vp(a,b)-i\vp(b,a)-\vp(b,b)=i(\vp(a,b)-\vp(a,b)^*)$$ then, $\vp(a,b)-\vp(a,b)^*=0_\C$, it thus follows that $\vp(a,b)=\pm\vp(a,b)^*$ hence $\vp(a,b)=0_\C$ and the equality holds. 
%	
%	
%	
%	% {\color{red}Indeed, in the inequality in the statement put $x=\vp(a,b)$; then $x^*x=0$ by (iii). Hence $x=0$, since $\|x\|_\C^2 =\|x^*x\|_\C$.} 
%\end{proof}

%\gbrem{ From now on we have to change the symbol of the norm}

%Let $\vp$ be a faithful $\C$-linear positive sesquilinear $\C$-valued map and consider the quasi norm defined in \eqref{eq: norm phi def}. 
It is easy to see that the following Cauchy-Schwarz inequality holds.
\begin{lem}\label{lem: Cauchy-Schwarz 1}  Let  $\vp$ be a $\C$-linear positive  sequilinear $\C$-valued map on $\X\times\X$, then
$\|\vp(a,b)\|_\C\leq\|\Lambda_\vp(a)\|_\vp \, \|\Lambda_\vp(b)\|_\vp$, for every $a,b\in\X$.
\end{lem}
%\begin{proof} Recall that if $0\leq x\leq y$, with $x,y\in \C^+$ then \ctrem{omettere?  detto sopra} $\|x\|_{\C}\leq\|y\|_{\C}$. Let $a,b\in\X$. From Lemma \ref{lem: CSI} it follows that 
%	\begin{eqnarray*}
%		\|\vp(a,b)\|_\C^2&=&	\|\vp(b,a)\vp(a,b)\|_\C\\ 
%		&\leq&\|\vp(a,a)\|_{\C} \,\|\vp(b,b)\|_\C=\|a+N_\vp\|_\vp^2 \, \|b+N_\vp\|_\vp \\ &=&\|\Lambda_\vp(a)\|_\vp \, \|\Lambda_\vp(b)\|_\vp.
%	\end{eqnarray*}
%	
%\end{proof}

%We have seen that, if $\vp$ is $\C$-linear, then the Cauchy Schwarz inequality holds. But there is at least another situation in which it occurs, as the following lemma shows.

\begin{rem} If  $\vp$ is faithful and satisfies the Cauchy-Schwarz inequality,  then $\|\cdot\|_\vp$ defined in \eqref{eq: norm phi semidef} is not only a  quasi norm, but  is a norm: $\|a\|_\vp=0$ implies that $a=0$ and the triangular  inequality holds true. In fact:%, by Lemma \ref{lem: Cauchy-Schwarz 1}:
\begin{eqnarray*}
\|a+b\|_\vp^2&=&\|\vp(a+b,a+b)\|_\C\leq\|\vp(a,a)\|_\C+2\|\vp(a,b)\|_\C+\|\vp(b,b)\|_\C\\&\leq&\|a\|_\vp^2+2\|a\|_\vp \, \|b\|_\vp+\|b\|_\vp^2=(\|a\|_\vp+\|b\|_\vp)^2.
\end{eqnarray*}
\end{rem}

%{\gbc  If $\X$ is complete w.r.to the norm, what could be the name, which sounds like "Hilbert C*-module?"}\\ 

\subsection{The case of a locally convex quasi *-algebra}

In this section we will consider a locally convex quasi *-algebra $(\A,\A_0)$ with unit $\id$ which is, at once,  a $\C$-bimodule.\\

% We will maintain in this new framework the same notations as in \cite{FT_book}.
\begin{defn} \label{px}
We 	denote by $\QA$ the set of all  positive sesquilinear $\C$-valued maps on $\A \times \A$ that satisfy a property of invariance:
\begin{itemize}
\item[(I)]
$\vp(ac,d)=\vp(c, a^*d), \quad \forall \ a \in \A, \ c,d \in \Ao.$
\end{itemize}
\end{defn}	
We maintain the same notations as before: then $\Lambda_\vp(a)$ will denote the coset in $\A/N_{\varphi}$, containing $a$.

\begin{rem}\label{rem: equival}
We recall that \begin{equation*}
	\lim_{n \to\infty}\vp(a_n,a_n)= 0_\C \Leftrightarrow \lim_{n \to\infty}\|\Lambda_\vp(a_n)\|_\vp=0.
\end{equation*}
\end{rem}
\begin{prop}\label{prop_nonsing}
\vspace{-1mm} Let  $\vp$ be a positive sequilinear $\C$-valued map on $\A\times\A$.
The following statements are equivalent:
\begin{itemize}
	\item[{\em (i)}]$\Lambda_\vp(\Ao)$ is dense in  $\widetilde{\A}$;
	\item[{\em (ii)}] If $\{a_n\}$ is a sequence of elements of $\A$ such that:
	\begin{itemize}
		\item[{\em (ii.a)}]$\vp(a_n,c){\to} 0_\C$, as $n \to\infty$, for every $c \in \Ao$;
		\item[{\em (ii.b)}]$\vp(a_n-a_m,a_n-a_m){\to} 0_\C$, as $n,m \to\infty$;
	\end{itemize}
	then,  $\displaystyle \lim_{n\to \infty}\vp(a_n,a_n) =0_\C$.
\end{itemize}
\end{prop}
\begin{proof}We proceed along the lines of \cite[Proposition 2.3.2]{FT_book}. \\$(i)\Rightarrow (ii)$ 
Let $\{a_n\}\subset{\A}$ be a sequence as required in $(ii)$. Then, by $(ii.b)$ and Remark \ref{rem: equival}, the sequence $\{\Lambda_\vp(a_n)\}$ is Cauchy in the complete space $\widetilde{\A}$. % indeed $$\lim_{n,m \to\infty}\|\vp(a_n-a_m,a_n-a_m)\|_\C= \lim_{n,m \to\infty}\|a_n-a_m+N_\vp\|_\vp^2=0.$$ 
Then there exists $\xi\in\widetilde{\A}$ such that $\lim_{n \to\infty}\|\Lambda_\vp(a_n)-\xi\|_\vp=0$. 
Now, by $(ii.a)$
$$\|\ip{\xi}{\Lambda_\vp(c)}_\vp\|_\C =\lim_{n \to\infty}\|\ip{\Lambda_\vp(a_n)}{\Lambda_\vp(c)}_\vp\|_\C= \lim_{n \to\infty}\|\vp(a_n,c)\|_\C = 0,$$ for all $c\in\A_0$,  hence $\ip{\xi}{\Lambda_\vp(c)}_\vp=0_\C$, $\forall c\in\A_0$ i.e., $\xi$ is orthogonal to  $\Lambda_\vp(\Ao)$  dense subset of $\widetilde{\A}$, thus $\xi=0$.  Finally,
$$\lim_{n \to\infty}\|\vp(a_n,a_n)\|_\C=\|\ip{\xi}{\xi}_\vp\|_\C=\|\xi\|_\vp^2=0$$ and this is equivalent to $\lim_{n \to\infty}\vp(a_n,a_n)=0_\C$.\\
$(ii)\Rightarrow(i)$ 
Let $\xi\in\widetilde{\A}$ be a vector which is orthogonal to  $\Lambda_\vp(\Ao)$ i.e. $$\ip{\xi}{\Lambda_\vp(c)}_\vp=0_\C, \quad \forall c\in\Ao.$$ Suppose that $\{a_n\}\subset{\A}$ is a sequence such that $\Lambda_\vp(a_n) \stackrel{\|\cdot\|_\vp}{\to} \xi$ i.e. \begin{equation}\label{eq: Cauchy seq}
	\|\Lambda_\vp(a_n)-\xi\|_\vp\to 0,\mbox{ as }n\to\infty.	
\end{equation} Then, $\{a_n\}$ fulfills $(ii.a)$, indeed, for every $ c\in\A_0$
$$0=\|\ip{\xi}{\Lambda_\vp(c)}_\vp\|_\C =\lim_{n \to\infty}\|\ip{\Lambda_\vp(a_n)}{\Lambda_\vp(c)}_\vp\|_\C= \lim_{n \to\infty}\|\vp(a_n,c)\|_\C,$$
hence $\vp(a_n,c){\to} 0_\C$, as $n \to\infty$, for every $c \in \Ao$;
%  both $$\ip{\xi}{\lambda_\vp(c)}_\vp =\lim_{n \to\infty}\ip{a_n+N_\vp}{c+N_\vp}_\vp= \lim_{n \to\infty}\vp(a_n,c) = 0_\C,\quad \forall c\in\A_0$$ 
$(ii.b)$ follows because $\{\Lambda_\vp(a_n) \}$ is a convergent sequence in a complete space, hence it is Cauchy in $\widetilde{\A}$ i.e., $\lim_{n,m \to\infty}\|\Lambda_\vp(a_n-a_m)_\vp\|_\vp=0$ which is equivalent of saying that $$\lim_{n,m \to\infty}\vp(a_n-a_m,a_n-a_m)= 0_\C.$$ 
Thus, by hypothesis $\{a_n\}$ is such that $\lim_{n \to\infty}\vp(a_n,a_n) = 0_\C$;  hence, $$\lim_{n \to\infty}\|\vp(a_n,a_n)\|_\C =\lim_{n \to\infty}\|\Lambda_\vp(a_n)\|_\vp = 0.$$
Comparing with \eqref{eq: Cauchy seq} we get 
$\|\xi\|_\vp=0_\C$, hence $\xi=0$. It follows that $\lambda_\vp(\Ao)$ is dense in  $\widetilde{\A}$.
\end{proof}
\begin{defn} We denote by $\IA$ the subset of $\QA$ satisfying one of the conditions (i)  or (ii) of Proposition \ref{prop_nonsing}. \end{defn}

\subsection{Examples}
Before going forth, we give  some examples of C*-valued positive sesquilinear maps.  We denote  	by $\B(\H)$ be the C*-algebra of bounded operators on $\H$.

\begin{ex}
Let $\D$ be a dense subspace of a Hilbert space $\H$ and consider the *-algebra $\LDb$. Let $\H_1\subset\D$ be a closed subspace of $\H$ and $P$ be the orthogonal projection onto $\H_1$.
Then $P\B(\H)P$ is a von Neumann algebra  which can be identified with a subspace of  $\LDH$. If $V$ is an operator in $P\B(\H)P$, let $$\vp:\LDb\times\LDb\to P\B(\H)P$$ be given by $$\vp(A,B)=V^*B^*AV.$$
Then $\vp$ satisfies our assumptions, but in general $\vp$ is not $P\B(\H)P$-linear. However, if $V=P$, then $\vp$ is $P\B(\H)P$-linear.\\ Consider the quasi *-algebra $(\LDH,\LDb)$. Let $V_i$ be positive bounded operators on $\H$. Define the sesquilinear form $\vp$ on this quasi *-algebra by %$$\vp:\LDb\times\LDb\to P\B(\H_1\oplus\dots\oplus\H_n)P$$ be given by $\H_i$ for every $i\in\{1,...,n\}$.
 $$\vp(A,B)=\sum_{i=1}^n\ip{Ax_i}{Bx_i}V_i, \quad x_1,...,x_n\in\D.$$ Then $\vp\in\IA$.
\end{ex}

\begin{ex}\label{ex:6} Let $\mathfrak{M}$ be a von Neumann algebra and $\rho$ a normal faithful finite trace on $\mathfrak{M}_+$.
	Consider the proper CQ*-algebra $(L^p(\rho),L^\infty(\rho))$ (see \cite{FT_book}). Let $W\in L^\infty(\rho)$ such that $W\geq0$ %and $\|W\|=1$. 
	For every  $t\in[0,\|W\|]$  consider the function $$f_t(s):=\begin{cases}s \mbox{ for }0\leq s \leq t\\
		t \mbox{ for }t\leq s \leq \|W\|.
	\end{cases}$$
	Then $\|f_t\|_\infty\leq\|W\|$ and for each $t_1,t_2$ it is $$\|f_{t_1}-f_{t_2}\|_\infty\leq |t_1-t_2|. $$ Moreover, $f_t\upharpoonright\sigma(W)\in C(\sigma(W)) $. Then:
	$$\|f_t(W)\|_\frac{p}{p-2}\leq \rho(\mathbb{I})\|f_t(W)\|_\infty=\rho(\mathbb{I})\|W\|.$$ Hence, 
	$f_t(W)\in L^\frac{p}{p-2}(\rho)$, for each $t\in[0,\|W\|]$. Consider the right multiplication operator $$R_W: X\in L^p(\rho)\to XW\in L^\frac{p-1}{p}(\rho).$$ Let $\vp:L^p(\rho)\times L^p(\rho)\to C([0,\|W\|])$ be given by $$\vp(X,Y)(t)=\rho(X(R_{f_t(W)}Y)^*).$$For each  $t\in[0,\|W\|]$, $\vp(\cdot,\cdot)(t)$ is a well-defined positive sesquilinear (scalar valued) form on $L^p(\rho)\times L^p(\rho)$ because $f_t(W)\in C^*(W)\subseteq L^\infty(\rho)$, with $C^*(W)$ the C*-algebra generated by $W$. Moreover, $f_t(W)$ is a positive operator. To see that $\vp(X,Y)$ is continuous, just observe that for $t_1,t_2\in[0,\|W\|]$ we have that \begin{eqnarray*}
		|\vp(X,Y)(t_1)-\vp(X,Y)(t_2)|&\leq& \|X\|_p\|Y\|_p\|\|f_{t_1}(W)-f_{t_2}(W)\|_\frac{p}{p-2}\\&\leq&\|X\|_p\|Y\|_p\rho(\mathbb{I})(t_1-t_2).
	\end{eqnarray*}In order to see that $\Lambda_\vp(L^\infty(\rho))$ is dense in $L^p(\rho)/N_\vp$, just observe that for each $t\in[0,\|W\|]$, we have that, for every sequence $\{X_n\}\subseteq L^\infty(\rho)$ and $X\in L^p(\rho)$, with $X_n\to X$:
	\begin{eqnarray*}
		|\vp(X_n-X,X_n-X)(t)|&\leq& \|X_n-X\|_p^2\,\|f_{t}(W)\|_\frac{p}{p-2}\\&\leq&\|X_n-X\|_p^2\,\|W\|\rho(\mathbb{I}).
	\end{eqnarray*}

	%  	Then \begin{align*}\|\vp_n(X,Y)\|_\infty&=\sup_{t\in[0,\|W\|]}\|X\|_p\|Y\|_p\|W_t\|_\infty\\ &=\|X\|_p\|Y\|_p\|W\|_\infty=\|X\|_p\|Y\|_p, \quad\forall n\in\mathbb{N}.\end{align*} Consider the Hilbert module $\ell_2(C([0,\|W\|]))$ and the sequence $\{v_n\}_n\subseteq C([0,\|W\|])$ be a sequence of non negative functions with $\|v_n\|_\infty\leq1$ for all $n\in\mathbb{N}$. Put $B^a(\ell_2(C([0,\|W\|])))$ the C*-algebra of all bounded, $C([0,\|W\|])$-linear, adjointable operators on $\ell_2(C([0,\|W\|]))$.  Set $\Phi:L^p(\rho)\times L^p(\rho)\to B^a(\ell_2(C([0,\|W\|])))$ defined by $$\Phi(X,Y)(f_1,f_2,...)=(\vp_1(X,Y)v_1f_1,\vp_2(X,Y)v_2f_2,...)$$for all $(f_1,f_2,...)\in \ell_2(C([0,\|W\|]))$. Clearly $\|\Phi(X,Y)\|\leq\|X\|_p\|Y\|_p$ for all $X,Y\in L^p(\rho)$ so $\lambda_\Phi(L^\infty(\rho))$  is dense  in $\lambda_\Phi(L^p(\rho)/N_\vp)$. The other properties are easily verified. Similar considerations apply to the CQ*-algebra $(L^2([0,\|W\|]), L^\infty([0,\|W\|]))$ where we let $$\vp_n(f,g)=\int_0^1k_n(x,y)f(y)\overline{g(y)}dy$$ and $k_n\in C([0,\|W\|]\times [0,\|W\|])$, $k_n\geq 0$, $\|k_n\|=1$ for all $n\in\mathbb{N}$.
\end{ex}

 \begin{rem}
		The previous example holds also in the case $p=2$ ($\frac{p}{p-2}=\infty$). In this case $\|f_t(W)\|_\frac{p}{p-2}=\|f_t(W)\|_\infty\leq\|W\|$ and
		\begin{eqnarray*}
			|\vp(X,Y)(t_1)-\vp(X,Y)(t_2)|&\leq& \|X\|_2\|Y\|_2(t_1-t_2),\quad t_1,t_2\in[0,\|W\|]\end{eqnarray*}
		\begin{eqnarray*}
			|\vp(X_n-X,X_n-X)(t)|&\leq&\|X_n-X\|_2^2\,\|W\|,
		\end{eqnarray*}for every sequence $\{X_n\}\subseteq L^\infty(\rho)$ and $X\in L^2(\rho)$, with $X_n\to X$.
\end{rem}

\begin{ex}\label{ex 8}
	Let $\mathfrak{M}$ be a von Neumann algebra and $\rho$ a normal faithful finite trace on $\mathfrak{M}_+$. %Consider the non commutative  space $L^p(\rho)$  and, for $p\geq2$ $$\mathcal{B}^p_+:=\{A\in L^\frac{p}{p-2}(\rho): A\geq0, \|A\|_{p/(p-2)}\leq1\}$$(if $p=2$ then $\frac{p}{p-2}=\infty$). 
	Let $W$ be as in the previous example. %{\gbc Example \ref{ex:6}}. %Then $WP_0=P_0WP_0$ is a positive element in the von Neumann algebra $P_0\mathfrak{M}P_0$. Let $f\in C(\sigma(WP_0))$ such that $f\geq0$ and $\|f\|_\infty\leq1$.
	Consider the weakly *-measurable operator valued function from $[0,\|W\|]$ into $\BH$. We shall also consider the space $L^2([0,\|W\|],\BH)$ with respect to the Gel'fand-Pettis integral (see \cite{Jocic}). Consider $A_t\in L^2([0,\|W\|],\BH)$ such that $A_t\geq0$ for a.e. $t\in [0,\|W\|]$ with the right multiplication operator and the function $f_t$ as in Example \ref{ex:6}. Define $$\vp:L^p(\rho)\times L^p(\rho)\to \BH$$ to be $$\vp(X,Y)=\int_0^{\|W\|}\rho(X(R_{f_t(W)}Y)^*) A_tdt.$$
	Then $\vp(X,Y)$ is well-defined since $$\ip{\rho(X(R_{f_t(W)}Y)^*) A_t h_1}{h_2}=\rho(X(R_{f_t(W)}Y)^*)\ip{A_t h_1}{h_2}$$is a measurable function of $t$ for every fixed $h_1,h_2\in\H$, hence the function $\rho(X(R_{f_t(W)}Y)^*) A_t$ is weakly *-measurable. Moreover, put $\vertiii{A_t}_2=\left\|\int_0^{\|W\|}A_t^*A_t dt\right\|^{1/2}=\left(\sup_{h\in\H, \|h\|\leq 1}\int_0^{\|W\|}\|A_t(h)\|^2 dt\right)^{1/2}$, since for every $h\in \H$, with $\|h\|\leq 1$ \begin{multline*}\int_0^{\|W\|}\hspace{-2em}|\rho(X(R_{f_t(W)}Y)^*)|^2\| A_t h\|^2dt\\\leq\sup_{t\in[0,\|W\|]}|\rho(X(R_{f_t(W)}Y)^*)|^2\int_0^{\|W\|}\| A_t h\|^2dt\\
		\leq \sup_{t\in[0,\|W\|]}|\rho(X(R_{f_t(W)}Y)^*)|^2\vertiii{A_t}_2^2<+\infty,\end{multline*} it follows that $\rho(X(R_{f_t(W)}Y)^*)A_t\in L^2([0,\|W\|],\BH)$. It is straightforward to check that $\vp(X,X)\geq0$ for all $X\in L^p(\rho)$ by using our choice of $A_t\geq0$; moreover it is  $$\vp(CX,Y)=\vp(X,C^*Y), \quad\forall X,Y\in L^\infty(\rho),\, C\in L^p(\rho).$$If now we take a sequence $\{X_n\}\subseteq L^\infty(\rho)$ and $X\in L^p(\rho)$, by the above argument we deduce that, for every $h\in\H$, with $\|h\|\leq1$, \begin{multline*}\int_0^{\|W\|}\hspace{-2em}(\rho((X_n-X)(R_{f_t(W)}(X_n-X))^*))^2\| A_t h\|^2dt\\\leq\sup_{t\in[0,\|W\|]}\rho((X_n-X)(R_{f_t(W)}(X_n-X))^*)^2\vertiii{A_t}_2^2\\\leq\sup_{t\in[0,\|W\|]}\|X_n-X\|_p^4\|f_t(W)\|^2_\frac{p}{p-2}\vertiii{A_t}_2^2\\\leq\|X_n-X\|_p^4\|W\|^2\rho(\mathbb{I})\vertiii{A_t}_2^2,\quad \forall n\in\mathbb{N}.\end{multline*} This shows that $\Lambda_\vp(L^\infty(\rho))$ is dense in $L^p(\rho)/N_\vp$.
\end{ex}

\begin{ex}
Consider $(L^p(\rho),L^\infty(\rho))$ and let$\{\Phi_n\}_n$ be a family of  invariant positive $\C$-valued sesquilinear maps   with $\Lambda_{\vp_n}(\Ao)$ dense in $\A/N_{\vp_n}$ and such that there exists $M>0$ for which $$\|\Phi_n(X,Y)\|_\C\leq M \|X\|_p\|Y\|_p$$ for all $X,Y\in L^p(\rho)$. Define now $$\Phi:L^p(\rho)\times L^p(\rho)\to \C$$ by $$\Phi(X,Y)=\sum_{n=1}^\infty x_n\Phi_n(X,Y)x_n^*,$$ for all $X,Y\in L^p(\rho)$ and $ \{x_n\}\subseteq\C $ such that $\sum_{n=1}^\infty \|x_n\|^2<\infty$. Then $$\|\Phi(X,Y)\|_\C\leq M \|X\|_p\|Y\|_p\sum_{n=1}^\infty \|x_n\|^2,\quad X,Y\in L^p(\rho). $$ It is easy to verify that $\Phi$ is an invariant positive $\C$-valued sesquilinear map with $\Lambda_{\vp}(\Ao)$ dense in $\A/N_{\vp}$.
\end{ex}

\section{Construction of *-representations}\label{sect_GNS_rep}

An important tool for the study of the structure of a locally convex quasi *-algebra $(\A,\A_0)$ is the the Gelfand--Naimark--Segal (GNS) construction for an invariant positive sesquilinear (ips) form on $\A\times \A$. %Therefore, it is natural to make the same construction 
The aim of this section is to extend this construction
starting from a positive sesquilinear $\C$-valued maps on $\A\times \A$ when $(\A,\A_0)$ is a locally convex quasi *-algebra with unit $\id$.%  {\gbc and, at once, a $\C$-bimodule.} % (with $\C$ a C*-algebra with unit with C*-norm $\|\cdot\|_\C$).  
%Through the representation, the $\C$-bimodule $\A$ can be put in connection with ${\mathcal L}\ad(\D_\vp,\H _\vp)$ as a bimodule. 

%When a partial multiplication is involved, as usual, we need to consider  positive sesquilinear $\C$-valued maps enjoying a certain {\em invariance} property: those which are in $\QA$; this way, it would be possible for $\ppi$ to have property $(ii)$ in Definition \ref{defn_starrepmod}. %For a quasi *--algebra $(\A,\Ao)$ this set of sesquilinear forms is exactly $\IA$. %An analogous GNS construction for positive {\em linear} functionals is also possible, under certain circumstances, and it will be discussed later (Theorem \ref{moreGNS}).

%When the quasi *-algebra is also a C*-module, we want the *-representation $\ppi$ of $\A$ to maintain this rich structure so we introduce the following
\begin{defn} \label{defn_starrepmod} %Let $\C$ be a C*-algebra with unit $e$ and C*-norm $\|\cdot\|_\C$,
Let $(\A,\A_0)$ be a quasi *-algebra with unit $\id$.   Let $\DD_\ppi$ be a dense subspace
of a certain quasi $B_\C$-space $\X$ with
%quasi norm is induced by  the 
$\C$-valued inner product $\ip{\cdot}{\cdot}_\X$.   A linear map $\ppi$ from $\A$ into $\LDXr$ is called  a \emph{*--representation} of  $(\A, \Ao)$,
if the following properties are fulfilled:
\begin{itemize}
\item[(i)]  $\ppi(a^*)=\ppi(a)^\dagger:=\ppi(a)^*\upharpoonright\DD_\ppi, \quad \forall \ a\in \A$;
\item[(ii)] for $a\in \A$ and $c\in \Ao$, $\ppi(a)\mult\ppi(c)$ is well--defined and {$\ppi(a)\mult \ppi(c)=\ppi(ac)$}.
\end{itemize}
We assume that for every *-representation $\ppi$ of $(\A,\A_0)$,  $\ppi(\id)={\idop_{\D_\ppi}}$, the
identity operator on  the space $\DD_\ppi$.\\
 The *-representation $\ppi$  is said to be\begin{itemize}
\item {\em closable} if there exists $\widetilde{\ppi}$ closure of $\ppi$ defined as $\widetilde{\ppi}(a)=\overline{\ppi(a)}\upharpoonright{\widetilde{\D}_\ppi}$ where $\widetilde{\D}_\ppi$ is the completion under the graph topology $t_\ppi$ defined by the seminorms $\xi\in\D_\ppi\to\|\xi\|+\|\ppi(a)\xi\|$, $a\in\A$, where $\|\cdot\|$ is the norm induced {by the inner product on $\D_\ppi$}
\item \emph{closed} if  $\DD_\ppi[t_\ppi]$ is complete. 
\item  \emph{cyclic} if there exits $\xi\in\D_\ppi$ such that $\ppi(\A_o)\xi$ is dense in $\X$ in its quasi norm topology.
\end{itemize}

%\gbc{If   $\A$ is  a $\C$-bimodule, a *-representation of $(\A, \Ao)$ is called  a \emph{modular
%*--representation} of  $(\A, \Ao)$	if moreover \begin{itemize} \ctrem{Do we need this?}
%\item[iii)] $\ppi(\C\id)$ is a C*-algebra with unit such that  $\LDXr$ is a $\ppi(\C\id)$-bimodule \footnote{$\C\id$ is the set of all products of an element of $\C$ times the identity $\id$ of $\A_0$.}.
%\end{itemize}}
\end{defn}

Let $(\A,\A_0)$  be a locally convex quasi *-algebra with unit $\id$.%,  with $\A$  a $\C$-bimodule. %and $\C$  a C*-algebra with unit $e$ and C*-norm $\|\cdot\|_\C$.\\
\begin{thm}\label{thm_rep}
\vspace{-1mm} Let  $\vp\in\QA$.
The following statements are equivalent:%\gbrem{check carefully if the representation is called modular when it is the case} \ctrem{I suggest not to introduce {\em modular}}
\begin{itemize}
\item[{\em (i)}]$\vp\in\IA$;
\item[{\em (ii)}] there exists a quasi $B_\C$-space $\X_\vp$ with $\C$-valued inner product  $\ip{\cdot}{\cdot}_{\X_\vp}$,
%  				Banach space $\X_\vp$ whose quasi norm is induced by  a $\C$-valued inner product $\ip{\cdot}{\cdot}_{\X_\vp}$, 
a dense subspace $\DD_\vp\subseteq\X_\vp$ and a closed cyclic *-representation $\ppi_\vp:\A\to{\mathcal L}^\dag(\DD_\vp,\X_\vp)$ such that $$\ip{\ppi_\vp(a)x}{y}_{\X_\vp}=\ip{x}{\ppi_\vp(a^*)y}_{\X_\vp}, \quad \forall x,y\in\DD_\vp, a\in\A$$ and such that $$\vp(a,b)=\ip{\ppi_\vp(a)\xi_\vp}{\ppi_\vp(b)\xi_\vp}_{\X_\vp}, \quad \forall  a,b\in\A.$$ %where $\xi_\vp$ is a cyclic vector for $\ppi_\vp$
\end{itemize}
\end{thm}
\begin{proof}The proof proceeds along the lines of that one of \cite[Proposition 2.4.1]{FT_book}.\\$i)\Rightarrow ii)$ Let $\vp\in\IA$. %, then $\Lambda_\vp(\Ao)$ is dense in  the 
The completion $\widetilde{\A}$ of $\Lambda_\vp(\A)$ is, as we have seen,  a quasi $B_\C$-space with quasi norm $\|\cdot\|_\vp$ induced by the quasi inner product $\ip{\cdot}{\cdot}_\vp$: $\|a\|_\vp=\|\ip{a}{a}_\vp\|_\C^{1/2}, \; a\in\widetilde{\A}.$  For any $a\in\A$ and $c\in\Ao$ put
$$\ppi^\circ_\vp(a)(\Lambda_\vp(c)):=\Lambda_\vp (ac).$$ %The linear operator $\ppi_\vp^\circ(a)$ is a well-defined one from $\lambda_\vp(\Ao)$ into $\widetilde{\A}$. 
Let  $c\in N_\vp$ and  $\{c_n\}\subset\Ao$ such that $$\vp(c_n-ac,c_n-ac)\to0,\quad\mbox{ as } n\to\infty.$$ By the invariance of $\vp$ 
$$\|\vp(ac,d)\|_\C^2=\|\vp(c, a^*d)\|_\C^2\leq2\|\vp(c,c)\|_\C\|\vp(a^*d,a^*d)\|_\C=0, \quad \forall d \in \Ao,$$ and by \eqref{eqn:Q3} we get \begin{multline*}
\|\vp(ac,ac)\|_\C\leq2(\|\vp(ac,c_n)\|_\C+\|\vp(ac,ac-c_n)\|_\C)\\\leq 4\|\vp(ac,ac)\|_\C^{1/2}\|\vp(ac-c_n,ac-c_n)\|_\C^{1/2}\to0, \mbox{ as }n\to\infty.
\end{multline*}Hence we get $\vp(ac,ac)=0_\C$. Thus $ac\in N_\vp$ and for every $a\in\A$ the operator $\ppi_\vp^\circ(a)$ is well-defined from $\lambda_\vp(\Ao)$ into $\widetilde{\A}$. Further, 
%%%%%%%%%%%%%%%%%%%%%%%%%%%%%%%%%%%%%%%%%%%%%%% 		
%  		\gbrem{for $a\in \A$ and $c\in \Ao$, $\ppi(a)\mult\ppi(c)$ is well--defined and {$\ppi(a)\mult \ppi(c)=\ppi(ac)$}: 	\begin{equation*}  				A_1 \mult A_2 = {A_1}\ad\x A_2,
%  			\end{equation*} defined whenever $A_2$ is a weak right multiplier of
%  			$A_1$ (we shall write $A_2 \in R^{\rm w}(A_1)$ or $A_1 \in L^{\rm w}(A_2)$), that is, whenever $ A_2 {\D} \subset
%  			{\D}({A_1}\ad\x)$ and  $ A_1\x {\D} \subset {\D}(A_2\x).$}
%%%%%%%%%%%%%%%%%%%%%%%%%%%%%%%%%%%%%%%%%%%%%%%
for every $a\in\A, c,d\in\Ao$ \begin{eqnarray*}\ip{\ppi^\circ_\vp(a)(\Lambda_\vp(c)) }{\Lambda_\vp (d)}_\vp&=&\vp(ac,d)=\vp(c,a^*d)\\&=&\ip{\Lambda_\vp(c)}{\Lambda_\vp(a^*d)}_\vp\\&=&\ip{\Lambda_\vp(c)}{\ppi^\circ_\vp(a^*)(\Lambda_\vp(d))}_\vp, \end{eqnarray*} hence $\ppi^\circ_\vp(a^*)=\ppi^\circ_\vp(a)^\dagger$
and for every $a\in\A, c,d,f\in\Ao$.
%%%%%%%%%%%%%%%%%%%%%%%%%%%%%%%%%%HERE%%%%%%
\begin{eqnarray*}\ip{\ppi^\circ_\vp(ac)(\cls{f})}{\cls{d}_\vp}&=&\vp(acf,d)=\vp(cf,a^*d)\\&=&\ip{\ppi^\circ_\vp(c)(\cls{f}}{\ppi^\circ_\vp(a^*)({\cls{d})}}_\vp
\end{eqnarray*} 
By the definition given in Remark \ref{rem_partialmult}, we conclude that $\ppi_\vp^\circ(a)\mult\ppi_\vp^\circ(c)$ is well-defined and therefore
$$\ppi_\vp^\circ(ac)=\ppi_\vp^\circ(a)\mult\ppi_\vp^\circ(c),\quad\forall a\in\A, c\in\Ao.$$ Hence $\ppi_\vp^\circ$ is a *-representation and $\ppi_\vp^\circ\upharpoonright\Lambda_\vp(\Ao)$ maps $\Lambda_\vp(\Ao)$ into itself.  

The operator $\ppi^\circ_\vp(a)$ is closable:  $\ppi^\circ_\vp(a^*)$ is adjointable, $\ppi^\circ_\vp(a^*)^*$ is a closed extension of $\ppi^\circ_\vp(a)$ (see Remark \ref{rem_closable}). 	
Denote by $\overline{\ppi^\circ_\vp(a)}$ its closure and let $\DD_\vp$ denote the completion of $\cls{\Ao}$ in the graph topology $t_\ppi$ % and by $\D_\vp^a$ its domain. Put $\D_\vp=\bigcap_{a\in\A}\D_\vp^a$ 
and for each $a\in\A$ let $\ppi_\vp(a)=\overline{\ppi^\circ_\vp(a)}\upharpoonright{\D_\vp}$. 
Then $\ppi_{\vp}$ is a closed *-representation of $(\A,\Ao)$. 
Finally, since $(\A,\Ao)$ has a unit $\id$, it follows that $\Lambda_\vp(\id)=\id+N_\vp$ is a cyclic vector and $\ppi_\vp(\id)=\mathbb{I}_{\D_\vp}$.\\
$ii)\Rightarrow i)$ 
To prove that the sesquilinear  $\C$-valued form $$\vp(a,b)=\ip{\ppi_\vp(a)\xi_\vp}{\ppi_\vp(b)\xi_\vp}_{\X_\vp}, \quad \forall  a,b\in\A$$ where $\xi_\vp$ is a cyclic vector for $\ppi_\vp$ is in $\IA$ it suffices to prove that it is positive, invariant and $\lambda_\vp(\Ao)$ is dense in $\widetilde{\A}$.
By definition it is positive. Since $\ppi_\vp$ is a *-representation we get that \begin{eqnarray*}
\vp(ac,d)&=&\ip{\ppi_\vp(a)\ppi_\vp(c)\xi_\vp}{\ppi_\vp(d)\xi_\vp}_\vp\\&=&\ip{\ppi_\vp(c)\xi_\vp}{\ppi_\vp(a^*)\ppi_\vp(d)\xi_\vp}_\vp=\vp(c,a^*d).
\end{eqnarray*} By hypothesis $\ppi_\vp(\A_0)\xi_\vp$ is dense in $\X_\vp=\widetilde{\A}$, hence for every $a\in\A$ there exists $\{c_n\}\subset\Ao$ such that $$\|\ppi_\vp(a)\xi_\vp-\ppi_\vp(c_n)\xi_\vp\|_\vp\to 0, \quad n\to\infty.$$Hence, $$\|\Lambda_\vp(a)-\Lambda_\vp(c_n)\|_\vp^2=\vp(a-c_n,a-c_n)=\|\ppi_\vp(a)\xi_\vp-\ppi_\vp(c_n)\xi_\vp\|_\vp^2\to 0$$ as $n\to\infty$. This implies that $\Lambda_\vp(\Ao)$ is dense in $\widetilde{\A}$ and concludes the proof.
\end{proof}
%{\gbc Is  the representation unique up to unitary equivalence? Yes, we could write the proof}
 \begin{defn}
		The triple $(\ppi_\vp, \Lambda_\vp,\X_\vp)$ constructed in Theorem \ref{thm_rep} is called
		the GNS construction for $\vp$ and $\ppi_\vp$ is called the GNS representation of $\A$
		corresponding to $\vp$.
	\end{defn}
	\begin{prop}\label{prop: uniq repre}
		Let $(\A,\Ao)$ be a quasi *-algebra with unit $\id$ and $\vp \in \IA$.
		Then, the GNS construction $(\ppi_\vp, \Lambda_\vp,\X_\vp)$ is unique up to unitary equivalence.
	\end{prop} 

If we consider normed quasi *-algebra $(\A[\|\cdot\|],\Ao)$, then the underlying *-algebra $\Ao$ is dense (in this norm) in $\A$, hence we get automatically the density of $\Lambda_\vp( \A_0)$ in  $\widetilde{\A}$ when $\vp$ is bounded.

\begin{cor}\label{3.1.13 book}Let $(\A[\|\cdot\|],\Ao)$ be a normed quasi *-algebra and $\vp\in\QA$ be such that $\vp$ is  bounded with respect to $\|\cdot\|$. Then, $\vp\in\IA$.\end{cor}\begin{proof}
If $\vp \in\QA$ is bounded, then the subspace $\Lambda_\vp(\Ao)$ is dense in  $\widetilde{\A}$. Indeed, if $a \in \A$,
there exists a sequence $\{c_n\}\subset\Ao$, such that $c_n\to a$	in $\A$ as $n\to\infty$. Then, we have
\begin{multline*}\|\Lambda_\vp(a)- \Lambda_\vp(c_n)\|_\vp^2 = \|\vp(a - c_n, a - c_n)\|_\C\\\leq\|\vp\|^2\|a- c_n\|^2\to 0, \mbox { as }	n\to\infty.\qedhere\end{multline*}
\end{proof}

\begin{cor}\label{cor: right Banach module}Let  $\vp\in\QA$ and  $\A$ be also a right module over $\C$ and let $\vp$ satisfy \eqref{eq: N_vp right module} then,  the quasi $B_\C$-space $\X_\vp$ in Theorem \ref{thm_rep} is also a  Banach right module over $\C$ and $\ppi_\vp(a)$ is  a $\C$-linear operator for all $a\in\A$.  	
\end{cor}
\begin{proof}
By Lemma	\ref{lem: 4 L}, $\A/N_\vp[\|\cdot\|_\vp]$ is a normed right C*-module over $\C$. The right multiplication by an element of $\C$ can be extended by continuity to the completion $\widetilde{\A}$ of $\A/N_\vp$, hence $\widetilde{\A}$ is a Banach right module over $\C$. Further, $\ppi_\vp^\circ(a)$ is a $\C$-linear operator for every $a\in\A$: \begin{multline*}
\ppi^\circ_\vp(a)(\Lambda_\phi(cx))=\Lambda_\vp(acx)=(\Lambda_\vp(ac))x\\=[\ppi^\circ_\vp(a)(\Lambda_\vp(c)]x,\quad\forall a\in\A, c\in\Ao, x\in\C.\qedhere
\end{multline*}
\end{proof}
 \begin{defn}
	The positive sesquilinear $\C$-valued map $\vp$ on $\A \times \A$ is called {\em admissible} if, for every $a\in\A$, there exists some $\gamma_a>0$ such that \begin{equation*}\label{eq: N_vp right module str}\|\vp(ac,ac)\|_\C\leq\gamma_a\|\vp(c,c)\|_\C, \quad \forall c\in \Ao.\end{equation*}
\end{defn}
\begin{rem}
If $\vp\in\QA$ is admissible, then the *-representation $\ppi_\Phi$ constructed from $\Phi$ is bounded. Indeed: $$\|\ppi_\vp(a)\Lambda_\vp(x)\|_\vp^2=\|\vp(ax,ax)\|_\C\leq\gamma_a\|\vp(x,x)\|_\C=\gamma_a\|\Lambda_\vp(x)\|_\vp^2, $$ for every $a\in\A, x\in\C$.
\end{rem}

\begin{cor}Let  $\vp$ be a $\C$-linear form in $\QA$ and  $\A$ be also a right module over $\C$.  Then,  the quasi $B_\C$-space $\X_\vp$ in Theorem \ref{thm_rep} is also a right Hilbert module over $\C$ and $\ppi_\vp(a)$ is a $\C$-linear operator for all $a\in\A$.  
\end{cor}\begin{proof}If $\vp$ is $\C$-linear, then the thesis follows from Lemma \ref{lem: Cauchy-Schwarz 1}.\end{proof}

 As an application of what we have seen until now, if $\A$ is a *--algebra with unit
$\id$ every  positive  linear $\C$-valued map $\omega$  on $\A$ (i.e., such that
$\omega(a^*a) \in\C^+$, for all $a \in \A$) is representable.
\begin{cor}\label{moreGNS alg}  Let $\A$ be a *--algebra with unit
$\id$ and let $\omega$ be a positive  linear $\C$-valued map on $\A$. Then, there
exists a quasi $B_\C$-space $\X_\vp$ whose quasi norm is induced by a $\C$-valued quasi inner product $\ip{\cdot}{\cdot}_{\X_\vp}$, a dense subspace $\D_\omega\subseteq\X_\vp$ and a closed cyclic *--representation $\ppi_\omega$ of $\A$ with domain $\D_\omega$, such that $$\omega(b^*ac)=\ip{\ppi_\omega(a) \Lambda_\omega(c)}{\Lambda_\omega(b)}_{\X_\vp},\quad\forall a,b,c\in\A.$$ Moreover, there exists  a  cyclic vector $\eta_\omega$,
such that $$ \omega(a)=\ip{{\ppi}_\omega(a)\eta_\omega}{\eta_\omega}_{\X_\vp}, \quad \forall \ a \in \A.$$ %This
The representation is unique up to unitary equivalence.\end{cor}
\begin{proof}
The thesis can be proved by applying Theorem \ref{thm_rep} and Proposition \ref{prop: uniq repre} to $\vp:\A\times\A\to\C$ defined as $\vp(a,b)=\omega(b^*a)$, for all $a,b\in\A$,  considering $\Ao=\A$. %Indeed, $\vp$ is positive and invariant and $\ppi(\A)\Lambda_\vp(e)$ is dense in $\X_\vp$. 
Indeed, $\vp$ is positive and invariant: $\vp(a,a)=\omega(a^*a)\in\C^+$ and $\vp(ac,d)=\omega(d^*(ac))=\omega((a^*d)^*c)=\vp(c,a^*d)$ for all $a,c,d\in\A$ and naturally $\Lambda_\vp(\A)=\A/N_\vp$ is dense in its completion.  
\end{proof}

\begin{rem}
If in addition $\A$ is a $\C$-bimodule and $$\|\omega(x^*a^*ax)\|_\C\leq\|\omega(a^*a)\|_\C \|x\|_\C^2, \quad\forall a\in\A, x\in\C$$then $\X_\vp$ is a right quasi Banach module over $\C$ and $\ppi_\omega(a)$ is $\C$-linear for all $a\in\A$.
\end{rem}

\begin{rem}
If in addition $\A$ is a $\C$-bimodule and  $$\omega(ax)=\omega(a)x, \quad\forall a\in\A, x\in\C$$then $\X_\vp$ is a right Hilbert $\C$-module and $\ppi_\omega(a)$ is $\C$-linear for all $a\in\A$.
\end{rem}
The following corollary gives a result of *-representability of $\C$-valued bounded linear positive map on $(\A,\Ao)$.
\begin{cor}\label{cor: 3.12}
	Let $(\A[\|\cdot\|],\Ao)$ be a unital normed quasi *-algebra and $\omega$ be a $\C$-valued bounded linear positive map on $(\A,\Ao)$ ($\omega(c^*c)\geq0$, for every $c\in\Ao$). If there exists $M>0$ such that $\|\omega(d^*c)\|_\C\leq M\|c\|\|d\|$, for all $c,d\in\Ao$, then   there
	exists a quasi $B_\C$-space $\X_\vp$ whose quasi norm is induced by a $\C$-valued quasi inner product $\ip{\cdot}{\cdot}_{\X_\vp}$, a dense subspace $\D_\omega\subseteq\X_\vp$ and a closed cyclic  *--representation $\ppi_\omega$ of $(\A,\Ao)$  with domain $\D_\omega$ and   cyclic vector $\eta_\omega$,
	such that $$ \omega(a)=\ip{{\ppi}_\omega(a)\eta_\omega}{\eta_\omega}_{\X_\vp}, \quad \forall \ a \in \A,$$ and $$\omega(b^*ac)=\ip{\ppi_\omega(a) \Lambda_\omega(c)}{\Lambda_\omega(b)}_{\X_\vp},\quad\forall a\in\A,\,\,\forall b,c\in\Ao.$$ 
	 The representation is unique up to unitary equivalence.
\end{cor}\begin{proof}
Define $\vp_0:(b,c)\in\Ao\times\Ao\to\vp_0(b,c)=\omega(c^*b)\in\C$. Then $\vp_0$ is a  bounded  positive sesquilinear $\C$-valued map on $\Ao\times\Ao$ and $$\vp_0(bc,d)=\vp_0(c,b^*d),\quad \forall b,c,d\in\Ao.$$ Since $\Ao$ is dense in $\A$, it is easy to prove that $\vp_0$ can be extended by continuity, to a  bounded  positive sesquilinear $\C$-valued map on $\A\times\A$. Hence $\Lambda_\vp(\Ao)$ is dense in $\A/N_\vp$ since $\vp$ is bounded and  $\Ao\subset\A$ densely. If $a\in\A$ and $\{c_n\}\subset\Ao$ with $c_n\to a$ as $n\to\infty$, then also $c_n^*\to a^*$   as $n\to\infty$; since $(\A,\Ao)$ is a normed quasi *-algebra we have also:$$c_n b\to ab, \mbox{ and } c_n^*c\to a^*c,\quad n\to\infty,\,\, b,c\in\Ao.$$ Hence, since $\vp$ is bounded, we get that $\vp$ is invariant because\begin{equation*}
	\vp(ab,d)=\lim_{n \to\infty}\vp(c_nb,d)=\lim_{n \to\infty}\vp(b,c_n^*d)=\vp(b,a^*d).
	\end{equation*} Therefore, by Theorem \ref{thm_rep}, $\vp$ is *-representable. Finally, if $a\in\A$ and $\{c_n\}\subset\Ao$ with $c_n\to a$ as $n\to\infty$ then \begin{equation*}
	\omega(a)=\lim_{n \to\infty}\omega(c_n)=\lim_{n \to\infty}\vp_0(c_n,\id)=\vp(a,\id).
\end{equation*} As for the uniqueness of the *-representation, it follows from Theorem \ref{thm_rep}.\end{proof}

\begin{rem}Assume that $\omega$ satisfies the assumptions of Corollary \ref{cor: 3.12}. Then, if $a\in\A$, for every sequence $\{c_n\}_n\subseteq\Ao$ with $c_n\to a$ as $n\to \infty$, we have that $$4\|\omega(\id)\|_\C\,\lim_{n \to\infty}\|\omega(c_n^*c_n)\|_\C\geq\|\omega(a)\|_\C^2.$$ Indeed, since $\|\omega(d^*c)\|_\C\leq M \|c\|\|d\|$ for all $c,d\in \Ao$,  then it is not hard to see that $\{\omega(c_n^*c_n)\}$ is a Cauchy sequence in $\C$, hence $\lim_{n \to\infty}\|\omega(c_n^*c_n)\|_\C$ exists. Moreover, by the boundedness of $\omega$, we have that $$\lim_{n \to\infty}\omega(c_n)=\omega(a).$$ By Corollary \ref{stinesp} we obtain that $$4\|\omega(\id)\|_\C\,\|\omega(c_n^*c_n)\|_\C\geq\|\omega(c_n)\|_\C^2, \quad\forall n\in\mathbb{N}.$$By taking the limits on both sides of the previous inequality, we get the desired one.\end{rem}
\begin{ex}
	Let $W$ be a positive operator in $L^\infty(\rho)$ and $f_t(W)$ be as in our previous examples. Since \begin{equation*}
		|\rho(Xf_t(W))|\leq\|X\|_p\|f_t(W)\|_{\frac{p-1}{p}}\leq\|X\|_p\rho(\mathbb{I})\|f_t(W)\|_\infty,\quad \forall X\in L^p(\rho)
	\end{equation*}it is not hard to see that the map $$\omega: A\in L^p(\rho)\to \omega(A):=\rho(Af_t(W))\in C([0,\|W\|])$$ is a well-defined ($\|(f_{t_1}-f_{t_2})(W)\|_\infty=|t_1-t_2|$)  bounded linear positive map on $(L^p(\rho),L^\infty(\rho))$ with values on the C*-algebra $C([0,\|W\|])$ and \begin{multline*}
	\|\omega(X^*Y)\|_\C\leq\|X\|_p\|Y\|_p\sup_{t\in[0,\|W\|]}\|f_t(W)\|_\frac{p}{p-2}\\=\|X\|_p\|Y\|_p\sup_{t\in[0,\|W\|]}\|f_t(W)\|_\infty\rho(\mathbb{I})\\=\|X\|_p\|Y\|_p\|W\|\rho(\mathbb{I}), \quad \forall X,Y\in L^\infty(\rho).
		\end{multline*} Similarly, given  $A_t\in L^2([0,\|W\|],\BH)$ such that $A_t\geq0$ for a.e. $t\in [0,\|W\|]$, we can consider the map $\Omega: L^p(\rho)\to\B(\H)$ given by $$\Omega(X)=\int_0^{\|W\|}\rho(Xf_t(W))A_t dt.$$Here we consider Gel'fand-Pettis integral and $A_t\in L^2([0,\|W\|],\B(\H))$. Since, for each $h\in\H$ with $\|h\|\leq1$
		\begin{multline*}\int_0^{\|W\|}\left|\rho(Xf_t(W))\right|^2\| A_t h\|^2dt\\\leq\|X\|_p\rho(\mathbb{I})\|W\|_\infty\int_0^{\|W\|}\hspace{0em} \|A_th\|^2dt \\\leq\|X\|_p\rho(\mathbb{I})\|W\|_\infty\vertiii{A_t}_2^2,
	\end{multline*}  it follows that $\rho(Xf_t(W))   A_t\in L^2([0,\|W\|],\B(\H)) $, for all $X\in L^p(\rho)$. Moreover, for all $h_1,h_2\in \H$ with $\|h_i\|\leq1$, $i\in\{1,2\}$, we obtain that \begin{multline*}
	\left|\ip{\int_0^{\|W\|}\rho(X^*Yf_t(W))A_tdt\, h_1}{h_2}\right|=\left|\int_0^{\|W\|}\ip{\rho(X^*Yf_t(W))A_t\, h_1}{h_2}dt\right|\\\leq\int_0^{\|W\|}\left|\ip{\rho(X^*Yf_t(W))A_t\, h_1}{h_2}\right|dt\\
	\leq\int_0^{\|W\|}\left\|\rho(X^*Yf_t(W))A_t\, h_1\right\|\|h_2\|dt\\	\leq\left(\int_0^{\|W\|}\|\rho(X^*Yf_t(W))A_t\, h_1\|^2dt\right)^{1/2}\|W\|^{1/2}\\\leq \|X\|_p \|	Y\|_p\|W\|\rho(\mathbb{I})\left(\int_0^{\|W\|}\|A_t h_1\|^2 dt\right)^{1/2}\|W\|^{1/2}\\\leq\|X\|_p \|	Y\|_p\|W\|^{3/2}\rho(\mathbb{I})\vertiii{A_t}_2,\end{multline*} hence, $$\left\|\int_0^{\|W\|}\rho(Xf_t(W))A_t\,dt\right\|\leq\|X\|_p \|	Y\|_p\|W\|^{3/2}\rho(\mathbb{I})\vertiii{A_t}_2, \,\,\forall X,Y\in L^p(\rho).$$It follows that $\Omega$ satisfies the conditions of Corollary \ref{cor: 3.12}.
\end{ex}
\bigskip
Let $\Phi \in \IA$ and $\vartheta$ a state on $\C$. Then $\phi:= \vartheta \circ \Phi$, (i.e., $\phi(a,b)=\vartheta(\Phi(a,b))$, for every $a,b\in \A$) is an invariant positive sesquilinear {\em form} on $\A \times \A$.
We have 
$$N_\phi=\{a\in \A: \, \phi(a,a)=0\}= \{a\in \A; \, \phi (a,a)\in N_\vartheta\}.$$
Moreover we have
$$ |\phi(a,a)|=|\vartheta(\Phi(a,a))|\leq \|\Phi(a,a)\|_\C.$$
This implies that the map
$$T_{\phi,\Phi}:\Lambda_\Phi(a) \in \X_\Phi \to \lambda_\phi (a) \in \H_\phi$$ is well-defined and bounded, where $\H_\phi$ is the Hilbert space $\H_\phi$, completion of $\A/N_\phi$ with respect to $\|\cdot\|_\phi$, with $\lambda_\phi (a)$  the coset containing $a$, $\|\lambda_\phi (a)\|_\phi= \phi(a,a)^{1/2}$.
From this we deduce that $\lambda_\phi (\Ao)$ is dense in  $\H_\phi$. Thus, a GNS *-representation constructed from  the  invariant positive sesquilinear (ips-)form $\phi$ is possible (see \cite[Proposition 2.4.1]{FT_book}). Let us denote it by $\pi_\phi$.
Then, we have for $a\in \A$ and $b\in \Ao$,
$$\pi_\phi(a)\lambda_\phi (b)=\lambda_\phi(ab)= T_{\phi,\Phi}\Lambda_\Phi(ab)= T_{\phi,\Phi}\Pi_\Phi(a)\Lambda_\Phi(b). $$
On the other hand,
$$ \pi_\phi(a)\lambda_\phi (b)= \pi_\phi(a)T_{\phi,\Phi}\Lambda_\Phi (b).$$
Therefore $$T_{\phi,\Phi}\Pi_\Phi(a)= \pi_\phi(a)T_{\phi,\Phi}, \quad \forall a\in \A.$$
Hence $\pi_\phi$ and $\Pi_\Phi$ are intertwined with bounded intertwining operator $T_{\phi,\Phi}$, (see \cite[Definition 1.3.1]{AnTr2014}).

\bigskip
{\bf{Acknowledgements:} } GB and CT acknowledge that this work has been done within the activities of Gruppo UMI Teoria dell’Approssimazione e Applicazioni and of GNAMPA of the INdAM.

\bibliographystyle{amsplain}

\end{document}